\documentclass[11pt]{amsart}
\usepackage{amscd,amsxtra,amssymb,mathrsfs, bbm}
\usepackage{stmaryrd, leftindex}
\usepackage{amscd,amsxtra,amssymb, bbm}
\usepackage[new]{old-arrows}
\usepackage{geometry}
        {\begin{quotation}\begin{center}\begin{em}}
        {\par\end{em}\end{center}\end{quotation}}

\newtheorem{theorem}{Theorem}[section]
\newtheorem{corollary}[theorem]{Corollary}
\newtheorem{lemma}[theorem]{Lemma}
\newtheorem{proposition}[theorem]{Proposition}

\theoremstyle{definition}
\newtheorem{definition}[theorem]{Definition}
\newtheorem{remark}[theorem]{Remark}
\newtheorem{summary}[theorem]{Summary}
\newtheorem{example}[theorem]{Example}

\DeclareMathOperator{\rad}{\mathsf{rad}}

\DeclareMathOperator{\GL}{\mathsf{GL}}

\DeclareMathOperator{\Aut}{\mathsf{Aut}}
\DeclareMathOperator{\End}{\mathsf{End}}

\newcommand{\kk}{\mathbbm{k}}
\newcommand{\CC}{\mathbb{C}}
\newcommand{\RR}{\mathbb{R}}
\newcommand{\NN}{\mathbb{N}}
\newcommand{\KK}{\mathbb{K}}
\newcommand{\LL}{\mathbb{L}}
\newcommand{\ZZ}{\mathbb{Z}}

\newcommand{\rightarrowdbl}{\longrightarrow\mathrel{\mkern-14mu}\rightarrow}
\newcommand{\twoheadarrow}{\rightarrow\mathrel{\mkern-14mu}\rightarrow}

\newcommand{\FF}{\mathbb{F}}

\renewcommand{\SS}{\mathbb{S}}

\newcommand{\HH}{\mathsf{H}}

\renewcommand{\HH}{\mathbb{H}}
\newcommand{\PP}{\mathbbm{P}}

\newcommand{\OO}{\mathbb{O}}

\newcommand{\Ad}{\mathsf{Ad}}

\newcommand{\lar}{\longrightarrow}

\newcommand{\llbrace}{(\!(}
\newcommand{\rrbrace}{)\!)}
\newcommand{\llbraket}{\langle\!\langle}
\newcommand{\rrbraket}{\rangle\!\rangle}

\def\oA{\bar{A}}		\def\oH{\bar{H}}		\def\oM{\bar{M}}
\def\sb{\subseteq}

\def\kron#1#2{\xymatrix@C=2em{{#1}\ar@/^3pt/[r]\ar@/_3pt/[r]&{#2}}}

\usepackage{tikz}
\usetikzlibrary{calc, arrows, positioning, shapes, fit, matrix, decorations}
\usetikzlibrary{decorations.shapes, decorations.pathreplacing}

\tikzset{
  decorate with/.style={decorate,decoration={shape backgrounds,shape=#1,shape size=1.5mm}},
   deco/.style={decorate with=dart},
   ordi/.style={draw,-stealth,  thick},
   conj/.style={dashed, draw, thick},
   ve/.style={circle, draw, thick, fill=blue!20, inner sep=1pt, outer sep=2pt, minimum size=7pt},
    dot/.style={fill=blue!10,circle,draw, inner sep=1pt, minimum size=5pt},
  dv/.style={star,star points=5,
star point ratio=2, draw, thick, fill=green!20, inner sep=1pt,outer sep=2pt,minimum size=7pt}
}

\tikzset{
    tbl5 nodes/.style={
        rectangle,
        execute at begin node=$,
       execute at end node=$,
       fill=blue!5,
        align=center,
        text depth=0.5ex,
        text height=2ex,
        inner xsep=0pt,
        outer sep=0pt,
           },
    tbl5/.style={
        matrix of nodes,
        row sep=-\pgflinewidth,
        column sep=-\pgflinewidth,
        nodes={
            tbl5 nodes
        },
        execute at empty cell={\node[draw=none]{};}
    }
  }

\input xy
\xyoption{all}

\title[Classification of real nodal orders]{Classification of real nodal orders}

\author{Igor Burban}
\address{
Universit\"at Paderborn\\
Institut f\"ur Mathematik \\
Warburger Stra\ss{}e 100 \\
33098 Paderborn \\
Germany
}
\email{burban@math.uni-paderborn.de}

\author{Yuriy Drozd}
\address{
Institute of Mathematics, National Academy of
Sciences of Ukraine, Tereschenkivska str. 3, 01024 Kyiv, Ukraine}
\email{y.a.drozd@gmail.com}


\begin{document}

\begin{abstract}
In this paper, we study properties of nodal orders defined over arbitrary base fields. In particular we give a classification of complete real nodal orders.
\end{abstract}

\maketitle
\section{Introduction}
Nodal orders are non-commutative generalizations of the 
$\kk$-algebra $\kk\llbracket x, y\rrbracket/(xy)$, where $\kk$ is a field.  This class of algebras was introduced by the second-named author in \cite{NodalFirst}. Assume that $\kk$ is algebraically closed. A characteristic property of nodal orders is that they are representation-tame, whereas all other orders are representation-wild \cite{NodalFirst}. In our joint work \cite{Nodal} it was proven that nodal orders are even derived representation-tame. 

This class of orders  appears in various representation-theoretic settings. For example, such orders arise  in the description of tame blocks of the categories of Harish-Chandra modules for real reductive groups \cite{BGG, Khoroshkin}. 
 Of major  interest is a global version of nodal orders given by (tame) non-commutative projective nodal curves; see for instance  \cite{bdnpdalcurves}.
 
Nodal orders over an algebraically closed field $\kk$ were classified in 
 \cite{Voloshyn}.  Properties of nodal orders over arbitrary fields $\kk$ were elaborated one step further in \cite{DrozdZembyk,  BurbanDrozdQuotients}.
 
This work provides  a full classification of nodal orders over the field of real numbers $\RR$.  Note that there are  three non-isomorphic commutative real nodal orders: $\CC\llbracket x, y\rrbracket/(xy)$, $\RR\llbracket x, y\rrbracket/(xy)$ and $\RR\llbracket x, y\rrbracket/(x^2 + y^2)$ (whereas $\CC\llbracket x, y\rrbracket/(xy)$ is the only commutative nodal order over $\CC$). One reason for a higher complexity of the combinatorics of real nodal orders is non-vanishing of the Brauer group $\mathsf{Br}\bigl(\RR\llbrace t\rrbrace\bigr)$, which leads to a broader variety of possible rational hulls  compared to the case of an algebraically closed field.  Another reason is that in comparison to the case of $\CC$, we have more types of irreducible semi-simple nodal pairs over $\RR$. Nonetheless, it turns out that real nodal orders are parametrized by appropriate discrete data, although the corresponding combinatorial pattern turns out to be much more complicated compared to the case of algebraically closed fields.

The plan of this paper is as follows.

\smallskip
\noindent
In Section \ref{S:NodalOrders}, we recall the definition of nodal orders and illustrate it  by numerous  ``typical''  examples. The main result of this section is Theorem \ref{T:Cartesian} describing nodal orders as appropriate ``gluings'' of their hereditary covers and semi-simple quotients. 

\smallskip
\noindent
In Section \ref{S:Hereditary}, we recall the classification of Harada \cite{Harada, Harada2} of hereditary orders over a complete discrete valuation ring. The main results are Theorem \ref{T:HereditaryAutom} and Corollary \ref{C:SummaryAutomorphisms}  describing the automorphism group of a complete hereditary order. 

\smallskip
\noindent
Complete real hereditary orders are classified in Section \ref{S:RealHereditary}. The main result here is a proof of the fact that there are precisely four non-isomorphic complete maximal local real orders: $\RR\llbracket t\rrbracket$, $\CC\llbracket t\rrbracket$, $\HH\llbracket t\rrbracket$ and $\CC\llbracket t\rrbracket^{\mathsf{tw}}$, where 
\begin{equation}\label{E:twistedseries}
\CC\llbracket z\rrbracket^{\mathsf{tw}}  = \Bigl\{\sum\limits_{k = 0}^\infty \lambda_k z^k \, \Big|\, \lambda_k \in \CC \; \mbox{\rm for all} \; k \in \NN_0\Bigr\} \;\mbox{\rm and}\;  z \lambda = \bar{\lambda} z \; \mbox{\rm for all} \; \lambda \in \CC.
\end{equation}
Finite-dimensional semi-simple nodal pairs over an arbitrary field $\kk$ are studied in Section \ref{S:SemisimpleNodal}. The corresponding classification is provided by Theorem \ref{T:SemiSimpleNodal}, which is one of the key statements  of this work.

\smallskip
\noindent
Finally,  real nodal orders are studied in Section \ref{S:RealNodal}. The corresponding classification is given by Theorem \ref{T:NodalClassification} and Theorem \ref{T:Final}, which are  the main results of this article. 

\smallskip
\noindent
\emph{List of notation}. For convenience of the reader we introduce now the most important
notation used in this paper.

\begin{itemize}
\item Let  $\Gamma$ be a ring. 
\begin{itemize}
\item We denote by $\Gamma^\circ$ the opposite ring of $\Gamma$,  by $\Gamma^\ast$ the group of multiplicative units of $\Gamma$ and by $\Aut(\Gamma)$ the group of ring automorphisms of $\Gamma$. 
\item For any $a \in \Gamma^\ast$ we denote by  $\Ad_a \in \Aut(\Gamma)$ the corresponding inner automorphism given by the rule $\Ad_a(x) = a x a^{-1}$ for any $x \in \Gamma$.
\item The Jacobson radical of $\Gamma$ is denoted by $\rad(\Gamma)$. 
\item For any $p \in \NN$ we denote by $M_p(\Gamma)$ the ring of matrices of size $(p \times p)$ with coefficients in $\Gamma$. 
\item If $M$ is a left $\Gamma$-module then for any $a \in \Gamma$ we have the corresponding left multiplication map
$M \stackrel{\lambda_a}\lar M, x \mapsto ax$. Similarly, if $M$ is a right $\Gamma$-module then we have the right multiplication map $M \stackrel{\rho_a}\lar M, x \mapsto xa$.
\end{itemize}
\item Let $K$ be a field.
\begin{itemize}
\item We denote by $\mathsf{Br}(K)$ the Brauer group of $K$. 
\item Next, $K\llbracket t\rrbracket$  is the $K$-algebra of formal power series with coefficients in $K$ and $K\llbrace t\rrbrace$ is the corresponding quotient field of formal Laurent series. 
\end{itemize}
\item As usual,  $\RR$ is  the field of real numbers and $\CC$ is the field of complex numbers, whereas $\HH$ is the  algebra of quaternions over $\RR$. Next, we put  $\KK = \RR\llbrace t\rrbrace$, $\LL = \CC\llbrace t\rrbrace$.
\item Unless explicitly stated otherwise, all orders in our paper are assumed to be \emph{semi-local} and \emph{complete}.
\end{itemize}

\smallskip
\noindent
\emph{Acknowledgement}. 
We are indebted to the anonymous referee for their helpful comments and valuable suggestions. The first-named author is grateful to Osamu Iyama for drawing his attention to \cite{Iyama} on quadratic orders.
The work of the first-named author was partially supported by the German Research Foundation
SFB-TRR 358/1 2023 -- 491392403.

\section{Basic features of nodal orders}\label{S:NodalOrders}

\noindent
Let $A$ be a ring, $A^\ast$ be its group of units and 
\begin{equation}\label{E:Jacobson}
J = \mathsf{rad}(A) = \left. \left\{a \in A \, \right| \, 1 + b a c \in A^\ast \; \mbox{\rm for all} \; b, c \in A \right\},
\end{equation}
be its Jacobson radical.
 Recall that  $A$ is  \emph{semi-local} if $\oA:= A/J$ is artinian (note that $\oA$ is automatically semi-simple). If $\oA$ is a skew field then $A$ is called \emph{local}.  If $A/J$ is a direct product 
  of skew fields, we call $A$ \emph{basic}. 
   Also, if $M$ is a left  $A$-module, we put 
  $\oM := M/JM$ and denote by $\ell_A(M)$ its length.   For a pair of semi-local rings $A\subseteq  H$ we denote 
  $$\ell^*(A,H)=\sup\Bigl\{\ell_A(H\otimes_A U) \, \big|\,  U \text{ is a simple left $A$-module}\Bigr\}.$$

  \begin{definition}\label{nod} 
  Let $A\subseteq H$ be a pair of semi-local rings.
  \begin{enumerate}
  \item[(i)]  This pair is called \emph{Backstr\"om} if $\rad(A)=\rad(H)$. 
  \item[(ii)]  It is called \emph{nodal} if it is Backstr\"om and  $\ell^*(A,H)\le2$.
  \end{enumerate}
  \end{definition}

\noindent  
  The following fact is obvious.
   \begin{proposition}\label{modrad} 
   If a pair $A\sb H$ is nodal, so is the pair $\oA\sb\oH$. Conversely, if $A\sb H$  is Backstr\"om and the pair $\oA\sb\oH$ is nodal,
   then the  pair $A\sb H$ is nodal too.
   \end{proposition}

\noindent
Let $R$ be  an excellent  reduced  equidimensional ring of {Krull dimension one}  and
$K := \mathsf{Quot}(R)$ be the corresponding total ring of fractions. 

\begin{definition}\label{D:Order} An $R$-algebra $A$ is an \emph{order} over $R$ if the following conditions are satisfied:
\begin{itemize}
\item[(i)]  $A$ is a finitely generated torsion free $R$-module.
  \item[(ii)] $A_K:= K \otimes_R A$ is a  semi-simple $K$-algebra.
\end{itemize}
\end{definition}

\noindent
Note the following easy but useful fact; see for instance \cite[Lemma 2.8]{bdnpdalcurves}.
\begin{lemma}
Let $R$ be as above, $R' \subseteq R$ be a finite ring extension and $A$ be an $R$-algebra. Then $A$ is an order over $R$  if and only if it  is an order over $R'$.
Moreover, if $K' := \mathsf{Quot}(R')$ then we have: $A_K \cong A_{K'}$.
\end{lemma}

\begin{definition}
A ring $A$ is an \emph{order} if 
 its center $R = Z(A)$ is an excellent  reduced equidimensional ring  of Krull dimension one, and $A$ is an order over $R$. If 
$K:= \mathsf{Quot}(R)$ then $Q(A):= K \otimes_R A$ is called the \emph{rational envelope} of  $A$.
\end{definition}

\begin{definition}\label{D:NodalOrders}
An order $A$ is called \emph{nodal} if its center is semi-local  and there exists a \emph{hereditary} overorder $H \supseteq A$ such that $(A, H)$ is a nodal pair.
\end{definition}

\begin{remark}
It is clear that an order  $A$ is nodal if and only if its radical completion $\widehat{A}$ is nodal. Moreover, given  a nodal order $A$,  a hereditary overorder $H$ from Definition \ref{D:NodalOrders} is 
 \emph{uniquely determined} and admits the following intrinsic description:
\begin{equation}\label{E:HereditaryCover}
H = \bigl\{a\in Q(A) \, \big| \, a J \subseteq J\bigr\} \cong \End_A(J),
\end{equation} where $J$ is viewed as a right $A$-module  (this is essentially a consequence of \cite[Theorem 39.11 and Corollary 39.12]{ReinerMO}).  For a nodal order $A$, the order
$H$ will be called the \emph{hereditary cover} of $A$. Hereditary orders form a special subclass of nodal orders. 
\end{remark}

\begin{definition}
Let $\kk$ be a field. A $\kk$-order is a  $\kk$-algebra $A$ with a semi-local center $R$ such that $R/\rad(R)$ is finite-dimensional over $\kk$ (note that $A$ is semi-local, too).  We say that such $A$ is \emph{complete} if $R$  is a complete $\kk$-algebra. Two $\kk$-orders are called isomorphic if they are isomorphic as $\kk$-algebras. 
\end{definition}

\begin{remark} Let $A$ be a complete $\kk$-order. 
Using Noether normalization, we can view  $A$ as  an order over the $\kk$-algebra of formal power series $\kk\llbracket t\rrbracket$ (appropriately embedded into the center of $A$). From now on,  $\kk$-orders are assumed to be complete.
\end{remark}

\begin{example}\label{Ex:NodalOrdersCx} The following $\kk$-algebras are nodal orders. 
\begin{enumerate}
\item[(i)] Let $A = \kk\llbracket x, y\rrbracket/(xy)$. 
Then in the above notation we have:  $Q(A) \cong \kk\llbrace x\rrbrace \times \kk\llbrace y\rrbrace$ and $H \cong \kk\llbracket x\rrbracket \times \kk\llbracket y\rrbracket$. The corresponding semi-simple nodal pair $(\bar{A}, \bar{H}\bigr)$ is given by the diagonal  embedding $\kk \stackrel{\mathsf{diag}}\lar \kk \times \kk$. 
\item[(ii)] Let $R = \kk\llbracket t\rrbracket$, $I = (t)$ and 
$$
M_2(R) \supset
A = 
\left. \left\{\left(
\begin{array}{cc}
a_{11} & a_{12} \\
a_{21} & a_{22} 
\end{array} \right)  \; \right| \; a_{11}(0) = {a}_{22}(0) \; \mbox{\rm and} \; a_{12}(0) = 0
\right\} \cong \kk\llbraket x, y\rrbraket/(x^2, y^2).
$$
Here, $\kk\llbraket x,y\rrbraket$ is the algebra of non-commutative formal power series in $x$ and $y$ and we identify $x$ with 
$\left(\begin{array}{cc}
0 & t \\
0 & 0
\end{array} \right)$ and $y$ with $\left(\begin{array}{cc}
0 & 0 \\
1 & 0
\end{array} \right)$. 
Then  we have:
$Q(A) = M_2\bigl(\kk\llbrace t\rrbrace\bigr)$ and 
$H = \left(\begin{array}{cc} R & I \\ R & R \end{array}\right)$. The induced nodal pair $(\bar{A}, \bar{H}\bigr)$ is given again by the diagonal  embedding $\kk \stackrel{\mathsf{diag}}\lar \kk \times \kk$.
\item[(iii)] Now, let $A = \left(\begin{array}{cc} R & I \\ I  & R \end{array}\right)$.
Then $Q(A) = M_2\bigl(\kk\llbrace t\rrbrace\bigr)$ and $H = M_2(R)$. This time, the semi-simple  nodal pair $(\bar{A}, \bar{H}\bigr)$ is given 
by the canonical   embedding $\kk \times \kk \stackrel{\mathsf{can}}\lar  M_2(\kk)$.
\item[(iv)] Consider now the order $A = \left(\begin{array}{ccc} R & I  &  I\\ I  & R & I \\ R & R & R \end{array}\right)$. Then  $Q(A) = M_3\bigl(\kk\llbrace t\rrbrace\bigr)$, whereas $H = \left(\begin{array}{ccc} R & R  & I \\ R  & R & I \\ R & R & R \end{array}\right)$. The  induced  nodal pair $(\bar{A}, \bar{H}\bigr)$ is given by the embedding
$
\bigl(\kk \times \kk\bigr) \times \kk \xrightarrow{\mathsf{can} \times \mathsf{id}} M_2(\kk) \times \kk.
$
In this case, $A$ is isomorphic to the arrow completion of the path algebra of the so-called Gelfand quiver
\begin{equation}\label{E:GelfandQuiver}
\xymatrix
{
- \ar@/^/[rr]^{a_{-}}  & &  \star \ar@/^/[ll]^{b_{-}}
 \ar@/_/[rr]_{b_{+}}
 & &
\ar@/_/[ll]_{a_{+}} +}  \quad a_{-} b_{-} = a_{+} b_{+}.
\end{equation}
For $\kk = \CC$, the category of finite-dimensional left $A$-modules is equivalent to the principal block of the category of Harish-Chandra modules for the Lie group 
$\mathsf{SL}_2(\RR)$; see \cite{Gelfand} as well as \cite{ABR, BurbanGnedin, BurbanDrozdRealGelfand} for some recent developments. 
\end{enumerate}
\end{example}

\noindent
 In the case when $\kk$ is algebraically closed, nodal $\kk$-orders were classified in \cite{Voloshyn}.
The next example illustrates some phenomena arising in the classification of real nodal orders, which are absent in the complex case.

\begin{example}\label{E:RealNodalExamples} Let $\KK = \RR\llbrace t\rrbrace$, 
$\LL = \CC\llbrace t\rrbrace$, 
$\OO = \RR\llbracket t\rrbracket$ and  $\PP = \CC\llbracket t\rrbracket$. The following $\RR$-algebras are nodal orders.
\begin{enumerate}
\item[(i)] Let $A = \RR\llbracket x, y\rrbracket/(x^2 + y^2)$. 
Then we have:  $Q(A) \cong \LL$ and $H \cong \PP$. The embedding $A \longhookrightarrow H$ is given by the assignment $x \mapsto t, y \mapsto i t$. The corresponding semi-simple nodal pair $(\bar{A}, \bar{H}\bigr)$ is given by the   embedding $\RR \longhookrightarrow \CC$. 
\item[(ii)] Consider $H = \CC\llbracket x\rrbracket^{\mathsf{tw}} \times \CC\llbracket y\rrbracket \supset A = 
\left\{(f, g) \; \big| \; f(0) = g(0) \right\}$; see (\ref{E:twistedseries}) for the definition of twisted power series. The corresponding semi-simple  nodal pair $(\bar{A}, \bar{H}\bigr)$ is given  by the diagonal  embedding $\CC \stackrel{\mathsf{diag}}\lar \CC \times \CC$.
\item[(iii)]  
Let $$
M_2(\OO) \supset
A = 
\left. \left\{\left(\begin{array}{cc}
a_{11} & a_{12} \\
a_{21} & a_{22} 
\end{array} \right)  \; \right| \; a_{11}(0) = a_{22}(0) \; \mbox{\rm and} \; a_{12}(0) = - a_{21}(0)   \right\}.$$ 
Then we have: $Q(A) = M_2(\KK)$ and $H = M_2(\OO)$. The nodal pair $(\bar{A}, \bar{H}\bigr)$ is given by the embedding
\begin{equation}\label{E:RegEmbCx}
\CC \stackrel{\delta}\lar M_2(\RR), \; a + ib \mapsto  
\left(\begin{array}{rc}
a & b \\
-b & a
\end{array}\right)\; \text{for} \; a, b \in \RR.
\end{equation}
 \item[(iv)]  
Consider  $$
M_2(\PP) \supset
A = 
\left. \left\{\left(\begin{array}{cc}
a_{11} & a_{12} \\
a_{21} & a_{22} 
\end{array} \right)  \; \right| \; a_{11}(0) = \bar{a}_{22}(0) \; \mbox{\rm and} \; a_{12}(0) = 0   \right\}.$$
Then  $Q(A) = M_2(\LL)$ and $H = \left. \left\{\left(\begin{array}{cc}
a_{11} & a_{12} \\
a_{21} & a_{22} 
\end{array} \right)  \; \right| \; a_{12}(0) = 0   \right\} \subset M_2(\PP)$. The nodal pair $(\bar{A}, \bar{H}\bigr) $ is given by the embedding
$
\CC \xrightarrow{\mathsf{diag}^\ast} \CC \times \CC, \; a \mapsto (a, \bar{a}).
$ In this case we have: 
$
A \cong \RR\llbraket x,y, i\rrbraket/(x^2, y^2, i^2 +1, xi + ix, yi + i y).
$
\item[(v)] In the same vein, let $$
M_2(\PP) \supset
A = 
\left. \left\{\left(\begin{array}{cc}
a_{11} & a_{12} \\
a_{21} & a_{22} 
\end{array} \right)  \; \right| \; a_{11}(0) = \bar{a}_{22}(0) \; \mbox{\rm and} \; a_{12}(0) =   -\bar{a}_{21}(0) \right\}.$$
Then  $Q(A) = M_2(\LL)$ and $H =  M_2(\PP)$. The nodal pair $(\bar{A}, \bar{H}\bigr) $ is given by the embedding
$$
\HH \cong \left. \left\{
\left(\begin{array}{rc}
a & b \\
-\bar{b} & \bar{a}
\end{array}\right) \right| a, b \in \CC
\right\} \subset M_2(\CC),
$$
where $\HH$ is the skew field of Hamilton quaternions. 

\item[(vi)] Consider 
\begin{equation}\label{E:RealGelfand}
M_3(\OO) \supset A = 
\left\{\left. 
\left(\begin{array}{ccc} a_{11} & a_{12} & a_{13} \\ a_{21} & a_{22} & a_{23} \\ a_{31} & a_{32} & a_{33}
\end{array}\right)  \; \right|\;  \begin{array}{l}a_{11}(0)= a_{22}(0),  a_{21}(0) = -a_{12}(0) \\
a_{13}(0) = 0 = a_{23}(0)
\end{array}\right\}.
\end{equation}
Then $Q(A) = M_3(\KK)$,
$
H = 
\left\{\left. 
\left(\begin{array}{ccc} a_{11} & a_{12} & a_{13} \\ a_{21} & a_{22} & a_{23} \\ a_{31} & a_{32} & a_{33}
\end{array}\right)  \; \right|\;  
a_{13}(0) = 0 = a_{23}(0)
\right\}$
and $\bar{H} = M_2(\RR) \times \RR$. The nodal pair $(\bar{A}, \bar{H})$ is given by the embedding 
$\CC \times \RR \xrightarrow{\delta \times \mathsf{id}} M_2(\RR) \times \RR$, where $\delta$ is given by (\ref{E:RegEmbCx}). 

The order $A$ defined by (\ref{E:RealGelfand}) is the so-called \emph{real Gelfand order}. The category of 
finite-dimensio\-nal left $A$-modules is equivalent to the principal block of the category of 
\emph{real} Harish-Chandra modules for the Lie group 
$\mathsf{SL}_2(\RR)$; see \cite{BurbanDrozdRealGelfand}.
\item[(vii)]  Finally, fix the embedding $\CC\longhookrightarrow \HH$ and consider 
$
A = \bigl\{a \in \HH\llbracket t\rrbracket \, \big| \, a(0) \in \CC \bigr\}.
$
Then $Q(A) = \HH\llbrace t\rrbrace$, $H = \HH\llbracket t\rrbracket$ and the nodal pair $(\bar{A}, \bar{H})$ is given by 
$\CC\longhookrightarrow \HH$.
\end{enumerate}
\end{example}

\begin{theorem}\label{T:Cartesian} Let $\kk$ be any field and $A$ be a nodal $\kk$-order. Let $J$ be the Jacobson radical of $A$, $H$ be the hereditary cover of $A$, $A \stackrel{\pi_\circ}\rightarrowdbl \bar{A} = A/J$ and $H 
\stackrel{\pi}\rightarrowdbl \bar{H} = H/J$ be the corresponding canonical projections and $A \stackrel{\imath_\circ}\longhookrightarrow H$ and $\bar{A} \stackrel{\imath}\longhookrightarrow \bar{H}$ be the corresponding canonical inclusions.  Then the commutative diagram 
\begin{equation}\label{E:NodalCartesian}
\begin{array}{c}
\xymatrix{
A \ar@{->>}[r]^-{\pi_\circ} \ar@{_{(}->}[d]_{\imath_\circ} & \bar{A} \ar@{^{(}->}[d]^-{\imath}\\
H \ar@{->>}[r]^-{\pi}  & \bar{H}
}
\end{array}
\end{equation}
is Cartesian (i.e.~both pull-back and push-out) in the category of $\kk$-algebras. 

Let $A'$ be another nodal $\kk$-order and $A \stackrel{\phi}\lar A'$ be an isomorphism of $\kk$-orders. Then the following diagram in the category of $\kk$-algebras 
\begin{equation}\label{E:Banalities}
\begin{array}{c}
\xymatrix{
H  \ar[d]_-{\tilde{\phi}}& A \ar@{_{(}->}[l]_{\imath_\circ} \ar@{->>}[r]^-{\pi_\circ} \ar[d]^-{\phi} & \bar{A} \ar[d]^-{\bar\phi}\\
H'  & A' \ar@{_{(}->}[l]_{\imath'_\circ} \ar@{->>}[r]^-{\pi'_\circ} & \bar{A}'
}
\end{array}
\end{equation}
is commutative, where $\tilde\phi$ and $\bar\phi$ are isomorphisms induced by $\phi$. 

Conversely, let $H$ be a hereditary $\kk$-order and $\Lambda$ a finite-dimensional semi-simple $\kk$-subalgebra of $\bar{H}$ such that $(\Lambda, \bar{H})$ is a semi-simple nodal pair over $\kk$. Let 
$\Lambda \stackrel{\jmath}\hookrightarrow \bar{H}$ be the corresponding embedding and 
\begin{equation}\label{E:Reconstruction}
A = \bigl\{a \in H  \, \big| \, \pi(a) \in \mathrm{Im}(\jmath)\bigr\}
\end{equation}
 be the  pull-back of $\pi$ and $\jmath$, so that we have a commutative diagram 
 \begin{equation}\label{E:NodalReconstruction}
\begin{array}{c}
\xymatrix{
A \ar@{->>}[r]^-{{\pi}_\ast} \ar@{_{(}->}[d]_{\jmath_\circ} & \Lambda \ar@{^{(}->}[d]^-{\jmath}\\
H \ar@{->>}[r]^-{\pi}  & \bar{H}
}
\end{array}
\end{equation}
  Then $A$ is a nodal $\kk$-order, $H$ is isomorphic to its hereditary cover, $\Lambda$ is isomorphic to its semi-simple quotient  and both diagrams (\ref{E:NodalCartesian}) and (\ref{E:NodalReconstruction}) can be naturally identified. 
\end{theorem}
\begin{proof} Let $A$ be a nodal $\kk$-order. Since $
A =  \bigl\{a \in H  \, \big| \, \pi(a) \in \mathrm{Im}(\imath)\bigr\}
$, it follows that (\ref{E:NodalCartesian}) is a pull-back diagram.  

Let $B$ be a $\kk$-algebra and $H \stackrel{f}\lar B$ and $\bar{A} \stackrel{g}\lar B$ be homomorphisms of $\kk$-algebras such that $g \pi_\circ = f \imath_\circ$. It follows that $f(J) = 0$. Hence, we get a unique homomorphism $\bar{H} \stackrel{\bar{f}}\lar B$ such that $\bar{f} \pi = f$. 
\begin{equation*}
\begin{array}{c}
\xymatrix{
A \ar@{->>}[r]^-{\pi_\circ} \ar@{_{(}->}[d]_{\imath_\circ} & \bar{A} \ar@{_{(}->}[d]^-{\imath} \ar@/^/[ddr]^g  & \\
H \ar@{->>}[r]^-{\pi}  \ar@/_/[drr]_f & \bar{H} \ar@{.>}[dr]|-{\bar{f}} & \\
& & B
}
\end{array}
\end{equation*}
But we also have: $\bar{f} \imath = g$, since $\bar{f} \imath  \pi_\circ = g \pi_\circ$ and $\pi_\circ$ is an epimorphism. This shows that (\ref{E:NodalCartesian}) is also a push-out diagram, as asserted. 

Next, let $A \stackrel{\phi}\lar A'$ be an isomorphism of nodal $\kk$-orders. It is clear that $\phi(J) = J'$ and that $\phi$ induces an isomorphism of the corresponding rational hulls $Q(A) \stackrel{\tilde\phi}\lar Q(A')$. It follows from the description (\ref{E:HereditaryCover}) of the hereditary cover of $A$ that we indeed get a commutative diagram (\ref{E:Banalities}). 

Now, let $H$ be a hereditary $\kk$-order, $J$ be its Jacobson radical  and $\Lambda  \subseteq  \bar{H}$ be  a semi-simple nodal pair over $\kk$. By assumption, there exists a finite extension of $\kk$-algebras $R = \kk\llbracket t\rrbracket \subseteq  Z(H)$. Then $H$ is an order over $R$. It follows from (\ref{E:Jacobson}) that  $t H \subseteq J$. Hence,  $\bar{H}$ and $\Lambda$ are finitely generated $R$-modules, on which  $t$ acts by zero. Let $A$ be given by (\ref{E:Reconstruction}). Then we have a short exact sequence of $R$-modules
$$
0 \lar A \stackrel{\left(\begin{smallmatrix}
\imath_\circ \\ \pi_\circ
\end{smallmatrix} \right)}\lar H \oplus \Lambda \xrightarrow{(\pi, -\imath)} \bar{H} \lar 0. 
$$
It is clear that $A$ is a finitely generated torsion free $R$-module. Let $K = Q(R) = \kk\llbrace t\rrbrace$. Since $K \otimes_R \Lambda = 0 = K \otimes_R \bar{H}$, it follows that 
$$
Q(A) = K \otimes_R A = K \otimes_R H = Q(H).
$$
As a consequence, $A$ is a $\kk$-order and $H$ is its hereditary overorder. 
Moreover, $J$ is a two-sided ideal in $A$ and $H/J \supseteq A/J = \Lambda$.
Since $\Lambda$ is semi-simple, we have $J' \subseteq J$, where $J'$ is the Jacobson radical of $A$. 
It remains to show the opposite inclusion $J \subseteq J'$. 

It follows from \cite[Proposition 5.22]{CurtisReiner} that 
\begin{equation}\label{E:DescrRadical}
J =  \left\{h \in H \, \big| \, h^n \in t H\; \mbox{\rm for  some} \; n \in \NN \right\}  \; \mbox{\rm and} \; J' = \left\{a \in A \, \big| \, a^n \in t A\; \mbox{\rm for  some} \; n \in \NN \right\}.
\end{equation}
Moreover, (\ref{E:DescrRadical}) implies that for any $k\in \NN$ we have: 
\begin{equation*}
J =  \left\{h \in H \, \Big| \, h^l \in t^{k+1} H\; \mbox{\rm for  some} \; l \in \NN \right\}.
\end{equation*}
Since $H/A$ is a finitely generated $R$-module and $K \otimes_R (H/A) = 0$, there exists $k \in \NN$ such that $t^k(H/A) = 0$. Let $a \in J$. Then $a \in A$ and $a^l \in t^{k+1} H \subseteq tA$ for some $l \in \NN$. Hence, $a \in J'$, as asserted.
\end{proof}

\begin{corollary}\label{C:Key}
Let $H$ be a hereditary $\kk$-order and $\Lambda \subseteq \bar{H}$ a $\kk$-subalgebra such that $(\Lambda, \bar{H})$ is a nodal pair. We denote by $A = H \vee  \Lambda$ the nodal order defined by the Cartesian diagram (\ref{E:NodalReconstruction}).  Let $(H', \Lambda')$ be another such datum and 
$A' = H' \vee  \Lambda'$. Then $A \cong A'$ as $\kk$-algebras if and only if there exists an isomorphism of $\kk$-algebras $H \stackrel{\psi}\lar H'$  such that $\bar\psi(\Lambda) = \Lambda'$, where $\bar{H} \stackrel{\bar\psi}\lar \bar{H}'$ is the induced isomorphism of the corresponding semi-simple quotients. 
\end{corollary}

\begin{summary}
The above Corollary \ref{C:Key} means that for a classification of  nodal $\kk$-orders, we have to provide  a description 
\begin{enumerate}
\item[(i)] of all possible hereditary $\kk$-orders $H$ as well as 
\item[(ii)] of all semi-simple nodal pairs $\Lambda \subseteq \bar{H}$ over $\kk$
\end{enumerate}
up to appropriate equivalences.
\end{summary}

\medskip
\noindent 
For a proof of the following result, we refer to 
\cite[Proposition 1.3]{DrozdZembyk}.
\begin{proposition}\label{P:Morita}
Let  $A$ be a nodal $\kk$-order and $A'$ be another $\kk$-order which is Morita equivalent to $A$. Then $A'$ is nodal, too. 
\end{proposition}

\section{Generalities on hereditary orders and 
their automorphisms}\label{S:Hereditary}
Let $R$ be a complete discrete valuation ring and $K = Q(R)$ be its quotient field. Next, let  $F$ be a finite-dimensional skew field over $K$ and $O$ be the maximal order in $F$. In fact, $O$ is the integral closure  of $R$ in $F$ (see \cite[Theorem 12.8]{ReinerMO}), so $O$ is actually unique. Moreover, if $L$ is the center of $F$, $O$ contains the integral closure of $R$ in $L$.
Hence, we can suppose that $F$ is central over $K$. Finally, let  $I$ be the Jacobson radical of $O$. We choose  $w \in I$
 such that $I = w O = O w$.

\begin{theorem}\label{T:StandardOrder}
For any $r \in \NN$ and 
$\vec{p} = \bigl(p_1, \dots, p_r\bigr) \in \NN^r$, consider the following order 
\begin{equation}\label{E:standardorder}
H(O, \vec{p}) :=
\left(
\begin{array}{ccc|ccc|c|ccc}
O & \dots & O & I & \dots & I & \dots & I  & \dots &  I \\
\vdots & \ddots & \vdots & \vdots &  & \vdots &  & \vdots & & \vdots \\
O & \dots & O & I & \dots & I & \dots & I & \dots &  I  \\
\hline
O & \dots & O & O & \dots & O & \dots & I & \dots &  I  \\
\vdots &  & \vdots & \vdots  & \ddots & \vdots &  & \vdots & & \vdots \\
O & \dots & O & O & \dots & O & \dots & I & \dots &  I  \\
\hline
\vdots &  & \vdots & \vdots & & \vdots & \ddots & \vdots &   & \vdots\\
\hline
O & \dots & O & O & \dots & O & \dots & O & \dots &  O  \\
\vdots &  & \vdots &  \vdots &  & \vdots &  & \vdots & \ddots & \vdots \\
O & \dots & O & O & \dots & O & \dots & O & \dots &  O  \\
\end{array}
\right) \subset M_{p }(F)
\end{equation}
over $R$, 
where the size
of the $i$-th diagonal block  is $(p_i \times p_i)$ for each $1 \le i \le r$ and $p := |\vec{p}| = p_1 + \dots + p_r$. Then the following statements are true.
\begin{enumerate}
\item[(a)] The order $H(O, \vec{p})$ is hereditary. Its center is $R$ and its rational hull is $M_{p }(F)$.
\item[(b)] Conversely, let $H$ be a hereditary order over $R$ in the simple $K$-algebra $M_{p}(F)$. Then there exist $r \in \NN$ and $\vec{p} \in \NN^r$ such that $|\vec{p}| = p$ and $H \cong H(O, \vec{p})$. More precisely, there exists $x \in \GL_{p}(F)$ such that
$H = \mathsf{Ad}_x\bigl(H(O, \vec{p})\bigr) = x \cdot H(O, \vec{p})\cdot x^{-1}$.
\item[(c)] Two orders $H(O, \vec{p})$ and $H(O, \vec{p}')$ are isomorphic if and only if $\vec{p}'$ can be transformed into $\vec{p}$ by a cyclic permutation. 
\end{enumerate} 
\end{theorem}

\noindent
\emph{Proofs} of these results are due to Harada \cite{Harada, Harada2}; see also 
 \cite[Theorem 39.14 and Corollary 39.24]{ReinerMO} for a textbook treatment.

\begin{remark} It is easy to see that the following properties of $H = H(O, \vec{p})$ are true. 
\begin{enumerate}
\item[(a)] Let $x = \left(\begin{array}{ccc} x_{11} & \dots & x_{1r} \\
\vdots & \ddots & \vdots \\
x_{r1} &\dots & x_{rr}
 \end{array}\right) \in H$, where $x_{ij}$ is a matrix of size $(p_i \times p_j)$ for any $1 \le i, j \le r$. Then $x \in H^\ast$ if and only if  $x_{ii} \in \mathsf{GL}_{p_i}(O)$ for all $1 \le i \le r$.   Similarly, $x \in \rad(H)$ if and only if 
 $x_{ii} \in M_{p_i}(I)$ for all $1 \le i \le r$. 
 \item[(b)] Let $\bar{O} = O/I$ be the residue skew field  of $O$. Then we have:
 $$
 \bar{H} = H/\rad(H) \cong M_{p_1}(\bar{O}) \times \dots \times M_{p_r}(\bar{O}).
 $$
 Moreover, the canonical map $H^\ast \lar \bar{H}^\ast$ is surjective. 
 \item[(c)] Let $y = \left(\begin{array}{ccc} y_{11} & \dots & y_{1r} \\
\vdots & \ddots & \vdots \\
y_{r1} &\dots & y_{rr}
 \end{array}\right) \in H$ and $\bar{y} = \left(\bar{y}_{11}, \dots, \bar{y}_{rr}\right)$ be its image in $\bar{H}$. Then for any $x \in H^\ast$ we have:
\begin{equation}\label{E:type1}
 \overline{\Ad_x(y)} = \bigr({\Ad_{\bar{x}_{11}}(\bar{y}_{11})}, \dots,  {\Ad_{\bar{x}_{rr}}(\bar{y}_{rr})}\bigr)
\end{equation}
and  any inner automorphism of $\bar{H}$ is induced by an inner automorphism of $H$. 
 \item[(d)] For any automorphism $\psi \in \Aut(O)$,  let $\bar\psi \in \Aut(\bar{O})$ be the induced automorphism. Since $\psi(I) = I$, we  obtain (abusing the notation) induced automorphisms $\psi \in \Aut(H)$ and $\bar\psi \in \Aut(\bar{H})$. Concretely, for  any $x \in H$ we have:
\begin{equation}\label{E:type2}
 \overline{\psi(x)} = \bigl(\bar\psi(\bar{x}_{11}), \dots,  
 \bar\psi(\bar{x}_{rr})\bigr).
\end{equation}
 \item[(e)] Let $1 \le t \le r$ be minimal with the property that 
$$
\vec{p} = \bigl(p_1, \dots, p_t; p_1, \dots, p_t;  \dots; p_1, \dots, p_t\bigr).
$$
For  $q = p_1 + \dots + p_t$ we put
\begin{equation}\label{E:Rotation}
\varrho = \varrho_{\vec{p}} = 
\left(
\begin{array}{ccccc}
0 & 0 & \dots & 0 & w \mathbbm{1}_q \\
\mathbbm{1}_q  & 0 & \dots & 0 & 0 \\
0 & \mathbbm{1}_q  & \dots & 0 & 0 \\
\vdots & \vdots & \ddots & \vdots & \vdots \\
0 & 0 & \dots & \mathbbm{1}_q  & 0
\end{array}
\right) \in M_p(F),
\end{equation}
where all blocks of $\varrho$ are square matrices of size $q$. Note that 
$p = l q$ and $r = lt$ for some $l \in \NN$.
Moreover,  $\varrho \in H \setminus H^\ast$ and $\Ad_\varrho(H) = H$. Since $\varrho^l = w \mathbbm{1}_p$, we conclude that 
$
\Ad_\varrho^l(x) = w x w^{-1}$ for any $x \in H$.

Let $\kappa \in \Aut(\bar{O})$ be the automorphism induced by $\Ad_w \in \Aut(O)$. 
Then the automorphism   
$\overline{\Ad}_\varrho \in \Aut(\bar{H})$ induced by $\Ad_\varrho$ is given by a twisted cyclic shift
\begin{equation}\label{E:type3}
(z_1, \dots, z_t, , \dots, z_{r-t+1}, \dots,  z_r) \mapsto 
\bigl(\kappa(z_{r-t+1}), \dots,  \kappa(z_r), z_1, \dots, z_t,  \dots\bigr).
\end{equation}
\end{enumerate}
\end{remark}

The goal of this section is to describe the group $\Aut(H)$ of all ring automorphisms  of the hereditary order $H = H(O, \vec{p})$. It contains a normal subgroup 
$\Aut_R(H)$ of all $R$-linear automorphisms of $H$ (i.e.~those ring automorphisms of $H$ which act trivially on its center). 
Presumably, many results of this section are well-known to experts. However, we were not able to find a reference for the key Corollary \ref{C:SummaryAutomorphisms}
below  and therefore state all preparatory results with detailed proofs.

Let $\Lambda = M_n(F)$ for some $n\in \NN$. For any  
$\phi \in \Aut(F)$, let $\widehat\phi \in \Aut(\Lambda)$ be the induced  automorphism.

\begin{lemma}\label{L:Prep1} For any  $\varphi \in \Aut(\Lambda)$ there exist $a \in \Lambda^\ast$ and $\phi \in \Aut(F)$ such that $\varphi = \Ad_a \cdot \widehat\phi$. 
\end{lemma}

\begin{proof} Let $S = F^n = \left(
\begin{array}{c}
F \\
\vdots \\
F
\end{array}
 \right)
 $ be the simple left $\Lambda$-module. For any $a \in \Lambda$ we have the corresponding left multiplication map $S \stackrel{\lambda_a}\lar S, \vec{v} \mapsto a \cdot \vec{v}$. 
 Analogously, for any $b \in F$ we get left 
  and right multiplication maps
 $$
 S \stackrel{\lambda_b}\lar S, \vec{v} \mapsto b \vec{v} \quad \mbox{\rm and} \quad 
  S \stackrel{\rho_b}\lar S, \vec{v} \mapsto\vec{v}  b. 
 $$
 Note that $F^\circ \stackrel{\rho}\lar \End_{\Lambda}(S), b \mapsto \rho_{b}$ is a ring isomorphism. 

Let $\varphi \in \Aut(\Lambda)$ be any automorphism. Then we get another simple left $\Lambda$-module $S^\sharp = F^n$, where the left action of $\Lambda$ on $F^n$ is given by the formula $a \ast \vec{v} = \varphi(a) \cdot \vec{v}$ for any $a \in \Lambda$ and $\vec{v} \in F^n$. Since a simple $\Lambda$-module is unique up to  isomorphisms, there exists an automorphism of the abelian group $S \stackrel{f}\lar S$ such that 
\begin{equation}\label{E:Fact1}
f  \lambda_a = \lambda_{\varphi(a)} f \quad \mbox{\rm for all}\;  a \in \Lambda.\end{equation} In other words, $S \stackrel{f}\lar S^\sharp$ is an isomorphism of $\Lambda$-modules and we get an induced ring isomorphism
$$
\End_{\Lambda}(S) \stackrel{\Ad_f}\lar \End_{\Lambda}(S^\sharp), \; g \mapsto  f g f^{-1}.
$$
As a consequence, we get a ring automorphism $F \stackrel{\phi}\lar F$ as well as the restricted automorphism $K \stackrel{\phi}\lar K$ such that the following diagram of rings and ring homomorphisms
\begin{equation}\label{E:Routine}
\begin{array}{c}
\xymatrix{
K  \ar[d]_-{{\phi}} \ar@{^{(}->}[r] &  F^\circ \ar[r]^-{\rho}  \ar[d]^-{{\phi}} & 
\End_{\Lambda}(S) \ar[d]^-{\Ad_f}\\
K   \ar@{^{(}->}[r] &  F^\circ \ar[r]^-{\rho}  & 
\End_{\Lambda}(S^\sharp) 
}
\end{array}
\end{equation}
is commutative. It follows that for any $b \in F$ we have: $f \rho_b = \rho_{\phi(b)} f$. Suppose now that $b \in K$. Then we have: $\rho_b = \lambda_b$, $\rho_{\phi(b)} = \lambda_{\phi(b)}$. Moreover,  (\ref{E:Fact1}) and (\ref{E:Routine}) imply that 
$$
\lambda_{\varphi(b)} f = f \lambda_b = \lambda_{\phi(b)} f.
$$
Since the maps $f$ and $\lambda$  are bijective, we conclude that $\varphi(b) = \phi(b)$ for all $b \in K$ (as usual, we identify an element of $K$ with the corresponding scalar matrix in $\Lambda$). Now we extend $\phi$ to a ring automorphism $\Lambda \stackrel{\widehat\phi}\lar \Lambda$. It follows that the composition $\Lambda \xrightarrow{\varphi \widehat{\phi}^{-1}} \Lambda$ acts trivially on $K$. By the Skolem--Noether theorem (see e.g.~\cite[Theorem  4.4.1]{DrozdKirichenko}) there exists $a \in \Lambda^\ast$ such that $\varphi \widehat{\phi}^{-1} = \Ad_a$, implying the statement. 
\end{proof}

\begin{lemma}\label{L:Prep2}
In the notation as above we have: $\Aut(F) \cong  \Aut(O)$.
\end{lemma}

\begin{proof} 
It is clear that any automorphism of $O$ induces an automorphism of its rational hull $F$, giving an injective map $\Aut(O) \longrightarrow \Aut(F)$. 
Conversely, since $K$ is the center of $F$, any element $\psi \in \Aut(F)$ restricts to  an automorphism $K \stackrel{\psi}\lar K$. It follows from \cite[Theorem 4.4.1]{EnglerPrestel} that $\psi(R) = R$. Since $O$ is the integral closure of $R$ in $F$ (see \cite[Theorem 12.8]{ReinerMO}), we conclude that $\psi(O) = O$. As a consequence, the map $\Aut(O) \longrightarrow \Aut(F)$ is surjective, hence bijective, as asserted. 
\end{proof}

From now on, let  $H = H(O, \vec{p})$ be a standard hereditary order given by (\ref{E:standardorder}). It is clear that any $\psi \in \Aut(O)$ induces a ring automorphism $\widetilde{\psi} \in \Aut(H)$, where $\widetilde{\psi}$ acts on the entries of the  elements of $H$ as $\psi$. 

\begin{lemma} For any $\vartheta \in \Aut(H)$ there exists $\psi \in \Aut(O)$ such that  $\vartheta \widetilde{\psi}^{-1} \in \Aut_R(H)$. 
\end{lemma}

\begin{proof} Let $\Lambda = M_p(F)$ be the rational hull of $H$ and $\widehat\vartheta \in \Aut(\Lambda)$ be the ring automorphism induced by $\vartheta$.  According to Lemma \ref{L:Prep1} and Lemma \ref{L:Prep2} there exists $\psi \in \Aut(O)$ such that 
$\widehat\vartheta \widehat\psi^{-1} \in \Aut_K(\Lambda)$, where 
$\widehat{\psi} \in \Aut(\Lambda)$ is the automorphism induced by $\psi$. Consider the automorphism $\vartheta \widetilde{\psi}^{-1} \in \Aut(H)$.
The corresponding induced morphism of the rational hull $\Lambda$ is $K$-linear, since it is equal to $\widehat\vartheta \widehat\psi^{-1}$. Hence,
$\vartheta \widetilde{\psi}^{-1} \in \Aut_R(H)$, as asserted. 
\end{proof}

Let $\varphi \in \Aut_R(H)$. The Skolem--Noether theorem implies that there exists $x \in \Lambda^\ast$ such that $\varphi(a) =  \Ad_x(a) = x a x^{-1}$ for all $a \in H$. 

\begin{definition}
Let $N_H = \bigl\{x \in \Lambda^\ast \, \big| \, \Ad_x(H) = H\bigr\}$ be the normalizer group of the order $H$. We say that $x, x' \in N_H$ are \emph{equivalent} if there exist $g, h \in H^\ast$ such that $x' = g x h$. 
\end{definition}

\noindent
Our goal is to describe the group $N_H$ explicitly.

\begin{theorem}\label{T:HereditaryAutom}
For any $x \in N_H$ there exist $d \in \ZZ$ and $0 \le k < l$ such that $x$ is equivalent to the matrix $w^{-d} \varrho^k$, where $w$ is a chosen generator of $I$ and $\varrho$ is given by (\ref{E:Rotation}).
\end{theorem}

\begin{proof} Let $\vec{e} = (1, \dots, 1) \in \NN^p$ and $H_\circ = H(O, \vec{e})$ be the corresponding hereditary order. Then $H_\circ$ and $\Ad_x(H_\circ)$ are two minimal hereditary suborders of $H$. According to \cite[Section 3]{Harada2}, there exists $y \in H^\ast$ such that 
$\Ad_{yx}(H_\circ) = \Ad_y\bigl(\Ad_x(H_\circ)\bigr) = H_\circ.$
Hence, replacing $x$ by an equivalent element $yx$ we may without loss of generality assume that $\Ad_{x}(H_\circ) = H_\circ$. 
Our next goal is to describe a normal form of an element of $N_{H_\circ}$ up to an equivalence relation. 

\smallskip
\noindent
\underline{Claim 1}. For any $x \in \Lambda^\ast$ there exist $y, z \in H_\circ^\ast$ such that $yxz$ is a monomial matrix, i.e.~a matrix having precisely  one non-zero entry in each row and column, which is  moreover a power of the element $w$.

\smallskip
\noindent
\underline{Proof}. Any  transformation of a matrix $x \in \Lambda^\ast$ into an equivalent one  is given by a sequence of elementary transformations of rows and columns of the following form.
\begin{enumerate}
\item For any $1 \le i < j \le p$, $a \in O$ and $b \in I$ we add
\begin{enumerate}
\item the $a$-multiple of the $i$-th row  to the $j$-th row,
\item the $a$-multiple of the $j$-th column to the $i$-th column,
\item the $b$-multiple of the $j$-th row to the $i$-th row, 
\item the $b$-multiple of the $i$-th column to the $j$-th column.
\end{enumerate}
\item For any $1 \le i \le p$ and $a \in O^\ast$ one can multiply the $i$-th row  by $a$, similarly the $i$-th column.
\end{enumerate}
For any $a \in F^\ast$ there exist uniquely determined  $d = \upsilon(a) \in \ZZ$ (the order of $a$) and $u \in O^\ast$ such that $a = u w^d$ (whereas 
$\upsilon(0) = \infty$).

For $1 \le i, j \le p$ let $x_{ij} \in F$ be the corresponding entry of $x \in M_p(F)$. Let $(i, j)$ be such that  $\upsilon(x_{ij})$ is the smallest possible and additionally $\upsilon(x_{ij}) < \upsilon(x_{ij'})$ and $\upsilon(x_{ij}) < \upsilon(x_{i'j})$ for any $i' < i$ and $j' > j$. Using the above transformations, one can eliminate all elements in the $i$-th row and $j$-th column but $x_{ij}$ and transform $x_{ij}$ to $w^d$, where $d = \upsilon(x_{ij})$. Proceeding recursively, we get the statement. \qed 

\smallskip
Recall that $U \subset F^p$ is  an $H_\circ$-\emph{lattice} if  $H_\circ U = U$, $U$ has finite rank over $R$  and $K \cdot U = F^p$. 

\smallskip
\noindent
\underline{Claim 2}. Let $U$ be an $H_\circ$-lattice  in $F^p$ and $x \in N_{H_\circ}$. Then $x U$ is also an $H_\circ$-lattice. Moreover, if $U$ and $U'$ are  $H_\circ$-lattices in $F^p$ then $U \subseteq U'$ if and only if $x U \subseteq x U'$. 

\smallskip
\noindent
\underline{Proof}. Since $H_\circ \cdot x U  = x H_\circ x^{-1} \cdot x U = x H_\circ  U = x U$, the statement follows. \qed 

\smallskip
\noindent
Consider the following descending chain of $H_\circ$-lattices:
$$
U_0 = 
\left(
\begin{array}{c}
O \\
O \\
\vdots \\
O \\
O
\end{array}
\right) \supset 
U_1 = 
\left(
\begin{array}{c}
I \\
O \\
\vdots \\
O \\
O
\end{array}
\right) \supset \dots \supset
U_{p-1} = 
\left(
\begin{array}{c}
I \\
I \\
\vdots \\
I \\
O
\end{array}
\right) \supset
U_{p} = w U_0 = 
\left(
\begin{array}{c}
I \\
I \\
\vdots \\
I \\
I
\end{array}
\right). 
$$
For any $d \in \ZZ$ we put: $U_{c+dp} := w^d U_c$. Since an $H_\circ$-lattice in $F^p$ is an indecomposable projective left $H_\circ$-module, such lattices form a totally ordered chain 
\begin{equation}\label{E:chain1}
\dots \subset U_{p+1} \subset U_p \subset \dots \subset U_1 \subset U_0 \subset U_{-1} \subset \dots
\end{equation}
Hence, for any $x \in N_{H_\circ}$ there exist unique $0 \le c < p$ and $d \in \ZZ$ such that $x U_0 = w^d U_c$. It follows that $w^{-d} x \in N_{H_\circ}$ is a monomial matrix and  $(w^{-d} x) U_0 = U_c$.

\smallskip
\noindent
\underline{Claim 3}. Let $y \in N_{H_\circ}$ be a monomial matrix such that 
$y U_0 = U_c$ for some $0 \le c < p$. Then we have: 
\begin{equation}\label{E:shape}
y = \left(
\begin{array}{cccccc}
0 & \dots & 0 & w & \dots & 0 \\
\vdots & \ddots & \vdots & \vdots & \ddots & \vdots \\
0 & \dots & 0 & 0 & \dots & w \\
1 & \dots & 0 & 0 & \dots & 0  \\
\vdots & \ddots & \vdots & \vdots & \ddots & \vdots \\
0 & \dots & 1 & 0 & \dots & 0 \\
\end{array}
\right),
\end{equation}
where there are precisely   $c$ rows of $y$  containing the element $w$. 

\smallskip
\noindent
\underline{Proof}. Since $y$ is a monomial matrix we have:
$y U_0 = 
\left(
\begin{array}{c}
(w^{d_1}) \\
\vdots \\
(w^{d_p})
\end{array}
\right),
$
where $w^{d_i}$ is the only non-zero element of  the $i$-th row of $y$. The assumption
$y U_0
= 
U_c$ for some  $0 \le c < p$ in particular implies that $d_i = 1$ for $1 \le i \le c$ and $d_i = 0$ for $c+1 \le i \le p$. 

\smallskip
\noindent
From  conditions $y U_i = U_{c+i}$ for $1\le i \le p-c$ we deduce that 
$y = \left(
\begin{array}{cc}
0 & \ast \\
\mathbbm{1}_{p-c} & 0
\end{array}
\right).
$
The equalities  $y U_{p-c +i} = U_{p+i}$ for $1 \le i \le c$ imply that $y$ is given by  (\ref{E:shape}), as asserted. \qed 

\smallskip

Since $H_\circ^\ast \subseteq H^\ast$, any two equivalent elements  $x, x' \in N_{H_\circ}$ are also equivalent in $N_H$. From what was said above it  follows that any element $x \in N_H$ is equivalent to $w^{-d} y$ for some $d \in \ZZ$, where $y$ is a  matrix of the form (\ref{E:shape}) for some $0 \le c < p$.  

Note that $H$-lattices in $F^p$  form the following subchain of the chain 
(\ref{E:chain1}):
$$
\dots \subset U_p \subset \dots \subset U_{p_1+p_2} \subset U_{p_1} \subset U_0 \subset \dots 
$$
Assume that $y \in N_H$ is of the form (\ref{E:shape}) for some $1 \le c < p$. Then $yU_0 = U_c$ with $c = p_1 + \dots + p_s$ for some $1 \le s \le r$. Moreover, we have $p_{s+i} = p_i$ for any $i \in \ZZ$ (where we put  $p_{r+ i} = p_{i}$ for any $i\in \ZZ$). Since the minimal period of the sequence $\vec{p}$ is $t$, we conclude that $s = t k$ for some $k \in \NN$; the bound  $1 \le k < l$ follows from $1 \le s < r = lt$. As a consequence,
$y = \varrho^{k}$ and the theorem is proven. 
\end{proof}

\begin{corollary}\label{C:SummaryAutomorphisms}
In the above notation, the automorphism group of the standard hereditary order $H = H(O, \vec{p})$ defined by (\ref{E:standardorder}) is generated by the following classes of elements. 
\begin{enumerate}
\item[(a)] Arbitrary inner automorphisms of $H$.
\item[(b)] The automorphisms induced by ring automorphisms of $O$.
\item[(c)] The automorphism $\Ad_\varrho$, where $\varrho$ is given by (\ref{E:Rotation}).
\end{enumerate}
The corresponding induced automorphisms of 
$
 \bar{H}  \cong M_{p_1}(\bar{O}) \times \dots \times M_{p_r}(\bar{O})
 $
 are the following.
 \begin{enumerate}
\item[(a)] Arbitrary inner automorphisms of $\bar{H}$.
\item[(b)] The automorphisms induced by the ring automorphisms of $\bar{O}$, which belong to the image of the canonical homomorphism $\Aut(O) \lar \Aut(\bar{O})$.
\item[(c)] Twisted cyclic shifts  given by the formula (\ref{E:type3}).
\end{enumerate}
Note at this place that if we make another  choice of a uniformizing element $\widetilde{w} \in I$ then $\widetilde{w} = u w$ for some $u \in O^\ast$. Hence, the corresponding twisted cyclic shift (\ref{E:type3}) is related to the original one by an inner automorphism. 
\end{corollary}

\begin{remark} In the case $O = R$, a description of the group $\Aut_R(H)$ is given in \cite[Theorem 5]{HaefnerPappacena}. Although the results of Corollary \ref{C:SummaryAutomorphisms} appear to be quite expected and standard, we were not able to find a proof of the general statement  in the existing literature. 
\end{remark}

\begin{lemma}\label{L:FieldExtensions}
Let $\kk$ be a field, $R = \kk\llbracket t\rrbracket$ and $K = \kk\llbrace t \rrbrace$. Let $K \subseteq K'$ be any finite field extension. Then there exists a finite field extension $\kk \subseteq \kk'$ and an isomorphism of 
$\kk$-algebras $K' \cong \kk'\llbrace w \rrbrace$. 
\end{lemma}

\begin{proof}  Let $R'$ be the integral closure of $R$ in $K'$. It follows from  \cite[Definition 23.1.1 and Theorem 23.1.5]{EGA} that the extension $R \subseteq R'$ is finite (regardless of the characteristic of $\kk$) and $Q(R') = K'$. Moreover, $R'$ is normal and has Krull dimension one, hence it is regular. Since $R \subseteq R'$ is finite and $R$ is complete and local, $R'$ is complete and semi-local.
Since $Q(R')$ is a field, $R'$ is local. Let $I$ be the maximal ideal of $R$ and $I'$ the maximal ideal  of $R'$. Then we get a finite field extension 
$\kk \cong R/I \subseteq R'/I' =: \kk'$. By the Cohen Structure Theorem (see e.g.~\cite[Theorem 7.7]{Eisenbud}) there exists an isomorphism of $\kk$-algebras 
$R' \cong \kk'\llbracket w\rrbracket$. As a consequence, $K' \cong \kk'\llbrace w \rrbrace$, as asserted. 
\end{proof}

\begin{theorem}\label{T:HeredOrders} Let $\kk$ be a field, $H$ be a hereditary $\kk$-order and $Z$ be the center of $H$.  Then the following statements are true.
\begin{enumerate}
\item[(a)] There is an isomorphism of $\kk$-algebras $Z \cong D_1  \times \dots \times D_t$, where $D_i = \kk_i\llbracket w\rrbracket$  and $\kk \subset \kk_i$ is a finite field extension for any $1 \le i \le t$. 
\item[(b)] Let $L_i = \kk_i\llbrace w\rrbrace$. Then there exist a finite-dimensional skew field $F_i$ over $L_i$ and $r_i \in \NN, \vec{p}_i \in \NN^{r_i}$ such that 
\begin{equation}\label{E:DecompHered}
H \cong H(O_1, \vec{p}_1) \times \dots \times H(O_t, \vec{p}_t),
\end{equation}
where $O_i$ is the maximal order in $F_i$.
\end{enumerate}
\end{theorem}

\noindent
\emph{Proofs} of all these results follow from \cite[Theorem 2.6]{Harada}, Lemma \ref{L:FieldExtensions} and  Theorem \ref{T:StandardOrder}.

\section{Finite-dimensional simple $\RR\llbrace t \rrbrace$-algebras and their maximal orders}\label{S:RealHereditary}
Let $\KK = \RR\llbrace t \rrbrace$ and $\Sigma$ be a finite-dimensional semi-simple $\KK$-algebra. Then we have: $\Sigma \cong \Sigma_1 \times \dots \times \Sigma_t$, where $\Sigma_i$ is a simple $\KK$-algebra for all $1 \le i \le t$. 
Let $\KK_i$ be the center of $\Sigma_i$. Then $\KK \subseteq \KK_i$ is a finite field extension. It follows from Lemma \ref{L:FieldExtensions} that 
$\KK_i$ (viewed as an $\RR$-algebra) is isomorphic either to $\KK$ or to $\LL = \CC\llbrace t \rrbrace$.

 Let  $\Sigma$ be a central simple $\LL$-algebra. Since the Brauer group
$\mathsf{Br}(\mathbb{L})$ of  $\LL$ is trivial (see for instance \cite[Corollary 6.3.5]{GilleSzamuely}), we  conclude that $\Sigma \cong M_n(\LL)$ for some $n \in \NN$. 

Suppose now that  $\Sigma$ is a central simple $\KK$-algebra. Then $\Sigma \cong M_n(\SS)$, where $\SS$ is a finite-dimensional skew field over $\KK$. Such skew fields are in bijection with the elements of the Brauer group
$\mathsf{Br}(\KK)$. First note that the Galois group of the extension 
$\LL/\KK$ is $G = \langle \sigma \, \big| \, \sigma^2 = e\rangle$, where
$
\sigma\left(\sum\limits_{l = l_0}^\infty \alpha_l t^l\right) = \sum\limits_{l = l_0}^\infty \bar{\alpha}_l t^l.
$
Since $\mathsf{Br}(\LL) = 0$, we have:
\begin{equation}
\mathsf{Br}(\KK) =  \mathsf{Br}(\LL/\KK) \cong H^2(G, \LL^\ast),
\end{equation}
where $\mathsf{Br}(\LL/\KK)$ denotes the relative Brauer group of the  extension $\LL/\KK$; see for instance \cite[Section 5.6]{DrozdKirichenko}. Let us now recall the description of the second Galois cohomology group.

 Let $L/K$ be a finite Galois field extension with the Galois group
$G = \mathsf{Gal}(L/K)$. The underlying group of compatible two-cocycles is defined as follows: 
$$
Z^2(G, L^\ast) =  \left. \left\{
G \times G \stackrel{\omega}\lar L^\ast
 \; \right| \;  \begin{array}{l}f\bigl(\omega(g, h)\bigr) \omega(f, gh)  = \omega(f, g) \omega(fg, h) \\ \omega(e, g) = 1 = \omega(g, e)\end{array}
\, \mbox{\rm for all} \, f, g, h \in G\right\}. 
$$
For any map $G \stackrel{\phi}\lar L^\ast$ such that $\phi(e) = 1$ we put:
$$
G \times G \stackrel{\partial_\phi}\lar L^\ast, \; \mbox{\rm where} \; \partial_{\phi}(f, g)  =  \phi(f) f\bigl(\phi(g)\bigr) \bigl(\phi(fg)\bigr)^{-1} \; \mbox{\rm for all} \; f, g \in G.
$$
Then we have: 
$$
B^2(G, L^\ast) =  \left. \left\{
G \times G \stackrel{\partial_\phi}\lar L^\ast
 \; \right| \;  \phi(e) = 1
\; \right\} \; \mbox{\rm and} \; H^2(G; L^\ast) = Z^2(G, L^\ast)/B^2(G, L^\ast).
$$
We refer to \cite[Theorem 5.6.6]{DrozdKirichenko} for a description of the group isomorphism 
\begin{equation}\label{E:RelativeBrauer}
H^2(G, L^\ast) \stackrel{\cong}\lar \mathsf{Br}(L/K).
\end{equation}

\begin{lemma}\label{L:GaloisReal} We have:
$
H^2\bigl(\mathsf{Gal}(\LL/\KK), \LL^\ast\bigr) \cong \ZZ_2 \times \ZZ_2. 
$
\end{lemma}

\begin{proof} Since $G = \mathsf{Gal}(\LL/\KK) = \{e, \sigma\}$, a compatible two-cocycle  
$\omega \in Z^2(G, \LL^\ast)$ is determined by its value $a = \omega(\sigma, \sigma) \in \LL^\ast$ satisfying the only condition $\sigma(a) = a$, i.e. $a \in \KK^\ast$. On the other hand, let $G \stackrel{\phi}\lar \LL^\ast$ be any map such that $\phi(e) = 1$. If  $\phi(\sigma) = b + ic$ with $b, c \in \KK$ then $\partial_\phi(\sigma, \sigma) = b^2 + c^2$. We claim that 
\begin{equation}\label{E:CohomologyComp}
H^2(G, \LL^\ast)  = \bigl\{\bar{1}, -\bar{1}, \bar{t}, -\bar{t} \bigr\}  \cong \ZZ_2 \times \ZZ_2.
\end{equation}
Indeed, any $a \in \KK^\ast$ can be uniquely written as $a = t^m u$ with $m \in \ZZ$ and $u \in \OO^\ast$. Hence, $\bar{a} = \overline{t^m} \bar{u}$ and 
\begin{equation*}
\overline{t^m} = \bar{t}^m = \left\{
\begin{array}{cl}
\bar{t} &\mathrm{if} \;  m \; \text{is odd} \\
\bar{1} &\mathrm{if} \;  m \; \text{is even}
\end{array}
\right.
\quad \text{and} \quad 
\bar{u} = \left\{
\begin{array}{rl}
\bar{1} &\mathrm{if} \;  u(0) >0  \\
-\bar{1} &\mathrm{if} \; u(0) < 0.
\end{array}
\right.
\end{equation*}
Hence, $\bar{a} \in \bigl\{\bar{1}, -\bar{1}, \bar{t}, -\bar{t} \bigr\}$ as asserted. Since any square in $\KK^\ast$ has positive leading coefficient and even valuation, we conclude that all elements of 
$\bigl\{\bar{1}, -\bar{1}, \bar{t}, -\bar{t} \bigr\}$ are pairwise distinct, implying (\ref{E:CohomologyComp}).
\end{proof}

\begin{corollary} Since all non-unit elements of $\mathsf{Br}(\KK)$ have order two, they correspond to quaternion algebras over $\KK$. Let $v \in 
\bigl\{1, -1, {t}, -{t} \bigr\}$. According to \cite[Section 5.6]{DrozdKirichenko} the corresponding central simple $\KK$-algebra $\Sigma_v$ has the following description: 
\begin{equation}
\Sigma_v = \HH(-1, v) = \left\langle i, j \, \big| \, i^2 = -1, j^2 = v, ij = - ji\right\rangle_{\KK}.
\end{equation}
Additionally, the following statements are true. 
\begin{enumerate}
\item[(a)] For $v = 1$ we have: $\Sigma_1 = M_2(\KK)$. Of course, 
$\bigl[\Sigma_1\bigr] = [\KK] \in \mathsf{Br}(\KK)$ is the neutral element. Therefore, abusing the notation, in what follows we put:  $\Sigma_1 := \KK$. 
\item[(b)] For $v = -1$ we have: $\Sigma_{-1} = \HH\llbrace t\rrbrace$, where $\HH$  is the skew field of real quaternions. 
\end{enumerate}
Note that $\Sigma_t$ and $\Sigma_{-t}$ are  not isomorphic as $\KK$-algebras but  are isomorphic as $\RR$-algebras. The corresponding isomorphism $\Sigma_t \lar \Sigma_{-t}$ is given 
by the assignment $i \mapsto i, j \mapsto j, t \mapsto -t$. As an $\RR$-algebra,  $\Sigma_{t}$ can be realized as the algebra of twisted formal Laurent  series
\begin{equation*}
\CC\llbrace z\rrbrace^{\mathsf{tw}} = \Bigl\{\sum\limits_{k = k_0}^\infty \lambda_k z^k \; \big|\, k_0 \in \ZZ,  \lambda_k \in \CC \; \mbox{\rm for all} \; k \in \ZZ\Bigr\}, \; \mbox{\rm where}\;  z \lambda = \bar{\lambda} z \; \mbox{\rm for all} \; \lambda \in \CC.
\end{equation*}
Here,  the isomorphism 
$\CC\llbrace z\rrbrace^{\mathsf{tw}} \cong \Sigma_t$ identifies $z$ with $j$. 
\end{corollary}

\smallskip
\noindent
For any  $v \in 
\bigl\{1, -1, {t}, -{t} \bigr\}$, let  us  now explicitly describe the corresponding maximal orders $O_v \subset \Sigma_v$ (recall again that $O_v$ is uniquely determined \cite[Theorem 12.8]{ReinerMO}).
Let $I_v$ be the Jacobson radical of $O_v$ and $F_v = O_v/I_v$ be the corresponding residue skew field. 

\begin{lemma} The following statements are true.
\begin{enumerate}
\item[(a)] If $v = 1$ then $O_1 = \RR\llbracket t\rrbracket$ and $F_1 \cong  \RR$.
\item[(b)] If $v = -1$ then $O_{-1} = \HH\llbracket t\rrbracket$ and $F_{-1} \cong  \HH$.
\item[(c)] If $v = \pm t$ then $O_{\pm t} = \left\langle i, j \, \big| \, i^2 = -1, j^2 = \pm t, ij = - ji\right\rangle_{\RR\llbracket t\rrbracket}$ and 
$F_{\pm t} \cong \CC$.
\end{enumerate}
\end{lemma}
\begin{proof} Note that for any $v \in 
\bigl\{1, -1, {t}, -{t} \bigr\}$ the algebra $O_v$ defined in the statement of  the lemma is an order in the  corresponding skew field $ \Sigma_v$.  We prove the maximality of $O_v$ for  $v = \pm t$; the remaining cases are similar.  First note that $I_v = O_v j = j O_v$ is the Jacobson radical of $O_v$. Since $\Sigma_v$ is a skew field, the morphism of left $O_v$-modules $O_v \rightarrow I_v, f \mapsto f j$ is injective, hence an isomorphism.   It follows that $I_v$ is projective as a left $O_v$-module.  By \cite[Theorem 39.1]{ReinerMO},
$O_v$ is hereditary. As $F_v = O_v/I_v \cong \CC$ is a field, $I_v$  is the unique maximal 
two-sided ideal, so $O_v$ is maximal by the Auslander--Goldman criterion (see \cite[Theorem 18.4]{ReinerMO}).
\end{proof}

\begin{remark}\label{R:RealOrders} Note that from the point of view of classification of hereditary $\RR$-orders, we may identify $O_{t}$ and $O_{-t}$.
We  have an isomorphism of $\RR$-algebras $O_{\pm t} \cong \CC\llbracket z\rrbracket^{\mathsf{tw}}$; see (\ref{E:twistedseries}) for the definition 
of $\CC\llbracket z\rrbracket^{\mathsf{tw}}$.
\end{remark}

\begin{summary}\label{SummaryRealHereditary} We recapitulate some key facts about real hereditary orders.
\begin{enumerate}
\item[(i)] By a classical theorem of Frobenius, $\RR$, $\CC$ and $\HH$ are the only finite-dimensional division algebras over $\RR$.
\item[(ii)] The only finite field extensions of the field $\RR\llbrace t\rrbrace$ up to $\RR$-linear  isomorphism are $\RR\llbrace w\rrbrace$ and  $\CC\llbrace w\rrbrace$; see Lemma \ref{L:FieldExtensions}.
\item[(iii)] It follows from the discussion made in this section that up to $\RR$-linear isomorphisms, the only  maximal local real orders are $\RR\llbracket t\rrbracket$, $\CC\llbracket t\rrbracket$, $\CC\llbracket t\rrbracket^{\mathsf{tw}}$ and $\HH\llbracket t\rrbracket$. Obviously, the  residue skew field of $\RR\llbracket t\rrbracket$ is $\RR$, for $\CC\llbracket t\rrbracket$ and  $\CC\llbracket t\rrbracket^{\mathsf{tw}}$ it is $\CC$ and for  $\HH\llbracket t\rrbracket$ it is $\HH$.
\item[(iv)] The orders $\RR\llbracket t\rrbracket$ and $\CC\llbracket t\rrbracket$ are commutative. It is easy to see that the center of $\CC\llbracket t\rrbracket^{\mathsf{tw}}$ is $\RR\llbracket t^2\rrbracket$.
Indeed, 
$$
z \left(\sum\limits_{k = 0}^\infty \lambda_k z^k\right) = \left(\sum\limits_{k = 0}^\infty \bar{\lambda}_k z^{k}\right) z \quad \text{and} \quad  \left(\sum\limits_{k = 0}^\infty \lambda_k z^k\right) i = i \left(\sum\limits_{k = 0}^\infty (-1)^k \lambda_k z^k\right)
$$
implying that any central element of $\CC\llbracket t\rrbracket^{\mathsf{tw}}$ must belong to $\RR\llbracket t^2\rrbracket$. Conversely, any element of $\RR\llbracket t^2\rrbracket$ is central, viewed as an element of 
$\CC\llbracket t\rrbracket^{\mathsf{tw}}$. Similar considerations show that the center of $\HH\llbracket t\rrbracket$ is $\RR\llbracket t\rrbracket$. For all orders $\RR\llbracket t\rrbracket$, $\CC\llbracket t\rrbracket$, $\CC\llbracket t\rrbracket^{\mathsf{tw}}$ and $\HH\llbracket t\rrbracket$,  the element $t$ generates the corresponding Jacobson radical.
\item[(v)] The field automorphism of $\CC$ induced by 
$\CC\llbracket t\rrbracket^{\mathsf{tw}} \xrightarrow{\mathsf{Ad}_t} \CC\llbracket t\rrbracket^{\mathsf{tw}}$ is 
\begin{equation}\label{E:type3tc}
\CC \stackrel{\kappa}\lar \CC, \; \lambda \mapsto \bar\lambda. 
\end{equation}
\item[(vi)] An explicit classification  of  hereditary $\RR$-orders is now provided by Theorem \ref{T:StandardOrder} and  Theorem \ref{T:HeredOrders} of Harada. A description of the corresponding automorphism groups  and their actions modulo the radical follows from Corollary \ref{C:SummaryAutomorphisms}. 
\end{enumerate}
\end{summary}

\section{Classification of semi-simple nodal pairs over an arbitrary field}\label{S:SemisimpleNodal}
Let $\kk$ be a field and $\phi, \widetilde\phi: \Lambda \lar \Gamma$ be two homomorphisms of finite-dimensional $\kk$-algebras. We say that $\phi$ and $\widetilde\phi$ are \emph{similar} (and denote $\phi \sim \widetilde\phi$) if there exists $b \in \Gamma^\ast$ such that 
$\widetilde\phi(a) = \mathsf{Ad}_b\bigl(\phi(a)\bigr) = b \phi(a) b^{-1}$ for any $a \in \Lambda$. 

Now assume that $\Lambda  \stackrel{\imath}\lar \Gamma$ is an injective homomorphism of finite-dimensional semi-simple $\kk$-algebras. Following Definition  \ref{nod}, we say that $\imath$ is a \emph{nodal embedding} if $\bigl(\imath(\Lambda), \Gamma\bigr)$ is a nodal pair. Of course, if  $\Lambda_k \stackrel{\imath_k}\lar \Gamma_k$ are nodal embeddings for $k = 1, 2$ then $\Lambda_1 \times \Lambda_2 \xrightarrow{\imath_1 \times \imath_2} \Gamma_1 \times \Gamma_2$ is a nodal embedding, too.

We shall need the following refinement of the notion of similarity introduced above. 
\begin{definition}
Let $\Gamma$ be a semi-simple $\kk$-algebra. Two nodal embeddings $\Lambda  \stackrel{\imath}\lar \Gamma$ and $\Lambda'  \stackrel{\imath'}\lar \Gamma$ are called \emph{similar} if there exists an isomorphism of $\kk$-algebras $\Lambda \stackrel{\nu}\lar \Lambda'$ and $b \in \Gamma^\ast$ such that 
the following diagram commutes:
$$
\xymatrix{
\Lambda \ar[rr]^-{\nu} \ar[d]_-{\imath} & &  \Lambda' \ar[d]^-{\imath'} \\
\Gamma \ar[rr]^-{\Ad_b} &  & \Gamma
}
$$
\end{definition}
Our next goal  is to classify all nodal embeddings into a given finite-dimensional semi-simple $\kk$-algebra $\Gamma$ up to similarity.

Let $K$ and $L$ be  finite-dimensional division $\kk$-algebras and $K \stackrel{\imath}\lar  L$ be a homomorphism of $\kk$-algebras.  Obviously, $L$ becomes a $(K$--$K)$-bimodule and 
\begin{equation}\label{E:Length}
\ell_K(L) := \ell_K\bigl(L_K\bigr) = \ell_K\bigl(_KL\bigr) = \dfrac{\dim_{\kk}(L)}{\dim_{\kk}(K)}.
\end{equation} For any $a \in L$ we have a map $
L \stackrel{\lambda_a}\lar L, \; x \mapsto ax. 
$
Obviously, $\lambda_{ab} = \lambda_a \lambda_b$ for any $a, b \in L$. Moreover, $\lambda_a$ is $K$-linear with respect to the right action of $K$ on $L$. 

Let $n = \ell_K(L)$ 
and  $\mathfrak{B}$ be a basis of $L_K$. Then we get 
an injective homomorphism of $\kk$-algebras $L \xrightarrow{\lambda(\imath, \mathfrak{B})}  M_n(K)$ called a \emph{regular embedding}, given by the composition 
$$
L \stackrel{\lambda}\lar \End_K(L_K) \stackrel{\beta^\mathfrak{B}}\lar M_n(K),
$$
where $\lambda(a) = \lambda_a$ for any $a \in L$ and $\beta^\mathfrak{B}(f) = [f]^\mathfrak{B}$ is the matrix of the endomorphism $f \in \End_K(L_K)$ with respect to the basis $\mathfrak{B}$. If $\mathfrak{B}'$ is another basis of $L_K$ then we obviously have: $\lambda(\imath, \mathfrak{B}) \sim 
\lambda(\imath, \mathfrak{B}')$. 

\begin{proposition}\label{P:regular1} Let $K \stackrel{\imath}\lar L$ be a homomorphism of finite-dimensional division $\kk$-algebras, $n = \ell_K(L)$, $\mathfrak{B} = (v_1, \dots, v_n)$ be a basis of $L_K$ and $L \xrightarrow{\lambda(\imath, \mathfrak{B})}  M_n(K)$ be the associated regular embedding. Then the following statements are true. 
\begin{enumerate}
\item[(i)] Let $K \stackrel{\tilde\imath}\lar L$ be another homomorphism of $\kk$-algebras such that $\tilde\imath \sim \imath$ and $\widetilde{\mathfrak{B}}$ be a basis of $L_K$ with respect to the module structure given by $\tilde\imath$. Then we have: $\lambda(\imath, \mathfrak{B}) \sim \lambda(\tilde\imath, \widetilde{\mathfrak{B}})$. 
\item[(ii)] Let $S = K^n$ and $M = M_n(K)$. Then $S$ has a standard  $(M$--$K)$-bimodule structure, which restricts to a $(L$--$K)$-bimodule structure given by the inclusion 
$\lambda(\imath, {\mathfrak{B}})$. Moreover, the map
$
S \stackrel{\zeta^{\mathfrak{B}}}\lar L, e_i \mapsto v_i \; \mbox{\rm for} \; 1 \le i \le n,
$
is an isomorphism of $(L$--$K)$-bimodules, where $(e_1, \dots, e_n)$ is the standard basis of $S=K^n$. 
\item[(iii)] The following diagram of $\kk$-algebras is commutative:
\begin{equation}
\begin{array}{c}
\xymatrix{
K^\circ \ar[r]^-\rho_-{\cong} \ar[d]_= & 
\End_M(S) \ar@{^{(}->}[r]^-{\jmath} & \End_L(S) \ar[d]^-{\widetilde\zeta^\mathfrak{B}}_-{\cong}  \\
K^\circ  \ar@{^{(}->}[r]^-\imath & L^\circ  \ar[r]^-\varrho_-{\cong}  & \End_L(L).
}
\end{array}
\end{equation}
Here $\jmath$ is the canonical restriction map, $\rho_a(v) = va$ for any $v \in S$ and $a \in K$, $\varrho_b(c) = cb$ for any $b, c \in L$ and $\widetilde\zeta^\mathfrak{B}(f)= \mathsf{Ad}_{\zeta^\mathfrak{B}}(f)$ for any $f \in \End_L(S)$. 
\item[(iv)] Let $C$ be the center of $K$, $\widetilde{C} = \imath(C) \subseteq L$ and $D$ be the center of $L$. 
\begin{enumerate}
\item If $\widetilde{C} \subseteq D$ then $L \xrightarrow{\lambda(\imath, \mathfrak{B})}  M_n(K)$ is a homomorphism of $C$-algebras. 
\item Suppose that $D \subseteq \widetilde{C}$ and $\widetilde{D} = \imath^{-1}(D) \subseteq K$. Then $L \xrightarrow{\lambda(\imath, \mathfrak{B})}  M_n(K)$  is a homomorphism of $D$-algebras. 
\end{enumerate}
\end{enumerate}
\end{proposition}

\begin{proof} (i) Let $b \in L^\ast$ be such that $\tilde{\imath}(a) = 
\mathsf{Ad}_b\bigl(\imath(a)\bigr)$ for all $a \in K$. First note that  $\widetilde{\mathfrak{B}} = \bigl(\mathsf{Ad}_b(v_1), \dots, \mathsf{Ad}_b(v_n) \bigr)$ is a basis of $L_K$ with respect to the $K$-module structure induced by $\tilde\imath$. Let $\lambda = \lambda(\imath, \mathfrak{B})$ and $\widetilde\lambda =  \lambda(\tilde\imath, \widetilde{\mathfrak{B}})$. A straightforward computation shows that
$$
\lambda(\tilde\imath, \widetilde{\mathfrak{B}})(a)  = X^{-1} \lambda(\imath, \mathfrak{B})(a) X \; \mbox{\rm for all} \; a \in L,
$$
where $X = \lambda(\imath, \mathfrak{B})(b)$. This shows that $\lambda(\imath, \mathfrak{B}) \sim \lambda(\tilde\imath, \widetilde{\mathfrak{B}})$, as asserted. 

\smallskip
\noindent
Proofs of the statements (ii) and (iii) are straightforward. 

\smallskip
\noindent
(iv) Suppose that $\widetilde{C} \subseteq D$. Then for any $a \in C$ and $1 \le j \le n$ we have $\imath(a) v_j = v_j \imath(a)$. It follows that  
$\lambda(\imath, \mathfrak{B})\bigl(\imath(a)\bigr) = a I$, where $I \in M_n(K)$ is the identity matrix. Hence, $L \xrightarrow{\lambda(\imath, \mathfrak{B})}  M_n(K)$ is a homomorphism of $C$-algebras, as asserted. 

In a similar vein, if $D \subseteq \widetilde{C}$ then $\lambda(\imath, \mathfrak{B})\bigl(\imath(a)\bigr) = a I$ for any $a \in \widetilde{D}$. As a consequence, $L \xrightarrow{\lambda(\imath, \mathfrak{B})}  M_n(K)$ is a homomorphism of $D$-algebras.
\end{proof}

\begin{proposition}\label{P:Key} 
Let $K, L$ be finite-dimensional division $\kk$-algebras and $L \stackrel{\delta}\lar   M_n(K)$ be a homomorphism of $\kk$-algebras for some $n \in \NN$. Then the  simple left $M_n(K)$-module  $S = K^n$ is also simple over $L$
if and only if $\dim_{\kk}(L) = n \cdot \dim_{\kk}(K)$. From now on assume this is the case. Then the following statements are true.  
\begin{enumerate}
\item[(i)] For any $0 \ne v \in S$ the map
$
L \stackrel{\xi_v}\lar S, \; a \mapsto \delta(a) v 
$
is an isomorphism of  $(L$--$K)$-bimodules. Next, there is a unique homomorphism of $\kk$-algebras $\gamma = \gamma(\delta, v): K \lar L$ (which is automatically injective) making the following diagram 
\begin{equation}\label{E:Return}
\begin{array}{c}
\xymatrix{
K^\circ \ar[r]^-\rho_-{\cong} \ar@{_{(}->}[d]_-\gamma & 
\End_M(S) \ar@{^{(}->}[r]^-{\jmath} & \End_L(S)   \\
L^\circ    \ar[rr]^-\varrho_-{\cong}  & &  \End_L(L) \ar[u]_-{\widetilde\xi_v}^-{\cong}
}
\end{array}
\end{equation}
commutative.
Here, $M = M_n(K)$, $\widetilde{\xi}_v(f) = \mathsf{Ad}_{\xi_v}(f)$ for any $f \in \End_L(L)$  and $\rho, \varrho$ and $\jmath$ are defined in the same way as in Proposition \ref{P:regular1}. Moreover, $\ell_K(L) = n$. 
\item[(ii)] For any $0 \ne v' \in S$ we have: $\gamma(\delta, v) \sim \gamma(\delta, v')$. More generally, let $L \stackrel{\tilde\delta}\lar   M_n(K)$ be another  homomorphism of $\kk$-algebras such that $\delta \sim \widetilde\delta$. Then $\gamma(\delta, v) \sim \gamma(\widetilde\delta, \widetilde{v})$ for any $0 \ne v, \widetilde{v} \in S$. 
\item[(iii)] Let $C$ be the center of $K$, $\widetilde{C} = \gamma(C) \subseteq L$ and $D$ be the center of $L$. 
\begin{enumerate}
\item If $\widetilde{C} \subseteq D$ then $L \stackrel{\delta}\lar   M_n(K)$ is a homomorphism of $C$-algebras. 
\item Suppose that $D \subseteq \widetilde{C}$ and $\widetilde{D} = \gamma^{-1}(D) \subseteq K$. Then  $L \stackrel{\delta}\lar   M_n(K)$ is a homomorphism of $D$-algebras. 
\end{enumerate}
\end{enumerate}
\end{proposition}

\begin{proof} Note that $S$ has a natural structure of an $(L$--$K)$-bimodule. It follows that
$$
\dim_{\kk}(S) = \ell_L(_LS) \cdot \dim_{\kk}(L) = \ell_K(S_K) \cdot \dim_{\kk}(K). 
$$
Since  $\ell_K(S_K) = n$, we conclude that  $\ell_L(_LS) = 1$  if and only if  $\dim_{\kk}(L) = n \cdot \dim_{\kk}(K)$.

\smallskip
\noindent
(i) The statement that $\xi_v$ is an isomorphism is obvious. 
Next,  we define $\gamma = \gamma(\delta, v): K \lar L$ as the composition of homomorphisms 
$$
\begin{array}{c}
\xymatrix{
K^\circ \ar[r]^-\rho  & 
\End_M(S) \ar@{^{(}->}[r]^-{\jmath} & \End_L(S)   \ar[r]^-{\widetilde\xi_v^{-1}} & 
\End_L(L) \ar[r]^-{\varrho^{-1}} & L^\circ.    
}
\end{array}
$$
Uniqueness of $\gamma$  is obvious. Note that commutativity of (\ref{E:Return}) is equivalent to the identity
\begin{equation}\label{E:KeyIdentity}
\bigl(\delta(b) v\bigr) a = \delta\bigl(b \gamma(a)\bigr) v \; \mbox{\rm for all} \; a\in K \; \mbox{\rm and} \;  b \in L.
\end{equation}
The formula $\ell_K(L) = n$ is a consequence of  (\ref{E:Length}). 

\smallskip
\noindent
(ii) Since $S$ is a simple $L$-module, for any $0 \ne v' \in S$ there exists a unique 
$c \in L$ such that $v' = \delta(c) v$. Let $\gamma'= \gamma(\delta, v')$. Replacing in (\ref{E:KeyIdentity}) the element  $b$ by $bc$ and writing the  analogous expression  of (\ref{E:KeyIdentity}) for $\gamma'$, we get:
$$
\bigl(\delta(b c) v\bigr) a = \delta\bigl(b c \gamma(a)\bigr) v \;
\mbox{\rm and} \; \bigl(\delta(b) \delta(c) v\bigr) a = \delta\bigl(b \gamma'(a)\bigr) \delta(c) v \; \mbox{\rm for all} \; a \in K, b \in L. 
$$
Since $L \stackrel{\xi_v}\lar S$ is a bijection, we get:  $c \gamma(a) = \gamma'(a) c$ for any $a \in K$. It follows that  $\gamma \sim \gamma'$, as asserted. 

Now, let 
$L \stackrel{\widetilde\delta}\lar   M_n(K)$ be another  homomorphism of $\kk$-algebras such that $\delta \sim \widetilde\delta$. Then there exists $d \in \mathsf{GL}_n(K)$ such that $d \delta(b) = \widetilde\delta(b) d$ for all $b \in L$. We put $\widetilde{v} := d v$ and $\widetilde\gamma = \gamma(\widetilde\delta, \widetilde{v})$. Then the identities
(\ref{E:KeyIdentity}) written both for $\gamma$ and $\widetilde\gamma$ imply that
$$
\delta\bigl(b \gamma(a)\bigr) v = \delta\bigl(b \widetilde\gamma(a)\bigr) v \; \mbox{\rm for all} \, a \in K, b \in L. 
$$
It follows that $\gamma = \widetilde\gamma$, implying the statement. 

\smallskip
\noindent
(iii) Suppose that $\widetilde{C} \subseteq D$. In order to show that $L \stackrel{\delta}\lar M_n(K)$ is a homomorphism of $C$-algebras, it is sufficient to prove  that $\delta\bigl(\gamma(a)\bigr) = a I$ for any $a \in C$, where $I \in M_n(K)$ is the identity matrix. 
For this note that   we have the following identities 
\begin{equation}\label{E:Hilfsaussage}
\delta\bigl(\gamma(a)\bigr) v = v a = (aI) v
\end{equation}
in the $(M_n(K)$--$K)$-bimodule $S$.
The first equality follows from (\ref{E:KeyIdentity}) for $b = 1$ and the second is a consequence of the fact that $a$ belongs to the center of $K$. Since
$\delta\bigl(\gamma(a)\bigr), a I \in \delta(L)$ and $L \stackrel{\xi_v}\lar S$ is an isomorphism, we conclude from (\ref{E:Hilfsaussage}) that $\delta\bigl(\gamma(a)\bigr) = a I$ for any $a \in  C$, as asserted.

In the case $D \subseteq \widetilde{C}$ we proceed in a similar way. By the same argument,   the identities (\ref{E:Hilfsaussage}) are true for any $a \in \widetilde{D}$, implying that 
$\delta\bigl(\gamma(a)\bigr) = a I$. This shows that   $L \stackrel{\delta}\lar   M_n(K)$ is a homomorphism of $D$-algebras. 
\end{proof}

\begin{lemma}\label{L:Cles} Let $K$ and $L$ be finite-dimensional division $\kk$-algebras such that $\dim_{\kk}(L) = 2 \dim_{\kk}(K)$  and $L \stackrel{\delta}\lar   M_2(K)$ be any homomorphism of $\kk$-algebras. For any $0  \ne v \in K^2$,  let $\gamma = \gamma(\delta, v): K \lar  L$ be the homomorphism of $\kk$-algebras constructed in Proposition \ref{P:Key}. As before, let $C$ be the center of $K$, $\widetilde{C} = \gamma(C)$ and $D$ be the center of $L$. Then the following statements are true.
\begin{enumerate}
\item[(i)] We have: either $D \subset \widetilde{C}$ or $\widetilde{C} \subseteq D$.
\item[(ii)] For any basis $\mathfrak{B}$ of $L_K$ let $\lambda = \lambda(\gamma, \mathfrak{B}): L \lar M_2(K)$ be the homomorphism of $\kk$-algebras constructed in Proposition \ref{P:regular1}. Then $\lambda$ and $\delta$ 
 are similar.
\end{enumerate}
\end{lemma}

\begin{proof} (i) First note that $\ell_K(L) = 2$. 
Next, suppose  that $D \setminus \widetilde{C} \ne \emptyset$. Take any element  $c \in D \setminus \widetilde{C}$. It is easy to see that $c \notin \widetilde{K} := \gamma(K)$, hence $L = \widetilde{K} \dotplus \widetilde{K}c$.

\smallskip
\noindent
\underline{Claim}. We have: $D = \widetilde{C} \dotplus \widetilde{C}c$. 

\smallskip
\noindent
We first show that $\widetilde{C} \dotplus \widetilde{C}c \subseteq D$. Indeed, let $a_1, b_1 \in \widetilde{C}$ be any elements. Then for any $d \in L$ there exist unique $a_2, b_2 \in \widetilde{K}$ such that $d = a_2 + c b_2$. Since  $[a_1, a_2] = [b_1, b_2]  = 0$, $[a_1, b_2] = [b_1, a_2] = 0$ and $c \in D$, we have:
$$
\bigl[a_1 + c b_1, d \bigr] = \bigl[a_1 + c b_1, a_2 + c b_2\bigr] = 0.
$$
It follows that $a_1 + c b_1 \in D$, as asserted. 

Conversely, let $b \in L$. Then we can find unique $x, y \in \widetilde{K}$ such that $b = x + yc$. For any $a \in L$ we have:
\begin{equation}\label{E:Decompositions}
ab = ax + ay c \quad \mbox{\rm and} \quad ba = xa + yca = xa + ya c,
\end{equation}
as $c$ is an element of the center of $L$. If $a \in \widetilde{K}$ then $ax, xa, ay, ya \in \widetilde{K}$. Assume now that $b \in D$.   We deduce  from the uniqueness of the decomposition of $ab = ba \in L$ in the $K$-basis $(1, c)$ of $L$ that  $ax = xa$ and $ay = ya$. It follows that $x, y \in \widetilde{C}$ and $b \in \widetilde{C} \dotplus \widetilde{C}c$. Hence, 
$D \subseteq  \widetilde{C} \dotplus \widetilde{C}c$ as asserted.

\smallskip
\noindent
(ii) First note that both $C$ and $D$ are finite field extensions of the base field $\kk$. Moreover, $L$ is a \emph{central} simple $D$-algebra and $M_2(K)$ is a \emph{central} simple $C$-algebra. 
\begin{itemize}
\item If $\widetilde{C} \subseteq D$ then both $\lambda, \delta: L \lar M_2(K)$ are homomorphisms of $C$-algebras; see Proposition \ref{P:regular1} and Proposition \ref{P:Key}. It follows from the Skolem--Noether theorem (see \cite[Theorem  4.4.1]{DrozdKirichenko}) that $\lambda$ and $\delta$ are similar. 
\item In the case $D \subset \widetilde{C}$, both $\lambda$ and $\delta$ are 
homomorphisms of $D$-algebras. By the ``dual'' Skolem--Noether theorem (see \cite[Theorem 4.4.4]{DrozdKirichenko}) $\lambda$ and $\delta$ are similar.
\end{itemize}
By the first part, either of the above cases is true. This proves that $\delta$ and $\lambda$ are similar. 
\end{proof}

\begin{corollary}\label{C:Regular} Let $K$ and $L$ be a pair of finite-dimensional division $\kk$-algebras such that $$\dim_{\kk}(L) = 2 \dim_{\kk}(K).$$ 
Then any homomorphism of $\kk$-algebras
$L \stackrel{\delta}\lar M_2(K)$ is similar to a regular embedding associated with  an appropriate  homomorphism of $\kk$-algebras $K \stackrel{\imath}\lar L$. Moreover, this correspondence  induces a bijection
$$
\left\{L \stackrel{\delta}\lar M_2(K)\right\}/\sim \; \; \Longleftrightarrow \; \;  \left\{K \stackrel{\imath}\lar L\right\}/\sim
$$
where the equivalence relations on both sides are given by the similarity relation. 
\end{corollary}
\begin{proof} This result is a consequence of Proposition \ref{P:regular1}, Proposition \ref{P:Key} and Lemma \ref{L:Cles}. 
\end{proof}

\begin{example}\label{E:Real} Let $\kk = \RR$. Then the set of  isomorphism classes of finite-dimensional division $\RR$-algebras is
$
\bigl\{\RR, \CC, \HH\bigr\}.
$
Moreover, up to conjugacy, there are precisely two embeddings of the type described in Corollary \ref{C:Regular}:
\begin{equation}\label{E:regular1}
\CC \cong \left. \left\{
\left(\begin{array}{cc}
a & -b \\
b & a
\end{array}\right) \right| a, b\in \RR
\right\} \subset M_2(\RR)
\end{equation}
and
\begin{equation}\label{E:regular2}
\HH \cong \left. \left\{
\left(\begin{array}{cc}
z &  w \\
-\bar{w} & \bar{z}
\end{array}\right) \right| z, w\in \CC
\right\} \subset M_2(\CC).
\end{equation}
In what follows, we shall call them just \emph{regular embeddings}. 
\end{example}

\smallskip
\noindent
Before we proceed with a classification of semi-simple nodal pairs over $\kk$, we recall the following result; see \cite[Theorem 3.5]{BurbanDrozdQuotients}.

\begin{theorem}\label{T:NodalCharact}
Let $\Gamma$ be a finite-dimensional semi-simple $\kk$-algebra and $\Lambda \subseteq \Gamma$ be a basic semi-simple $\kk$-subalgebra. Then the following statements are equivalent:
\begin{enumerate}
\item[(i)] $\Gamma$ is generated by at most two elements, viewed as a left $\Lambda$-module. Equivalently, there exists a surjective homomorphism of $\Lambda$-modules
$\Lambda^{{2}}  \twoheadarrow  \Gamma$.
\item[(ii)] $(\Lambda, \Gamma)$ is a semi-simple nodal pair over $\kk$. 
\end{enumerate}
\end{theorem}

\smallskip
\noindent
After all the necessary preparation work has been done, we can prove  the main technical result of this paper. 

\begin{theorem}\label{T:SemiSimpleNodal}
Let $\kk$ be any field and $\Gamma  = \Gamma_1 \times \dots \times \Gamma_m$ be a finite-dimensional semi-simple $\kk$-algebra with  $\Gamma_j = M_{r_j}(F_j)$ for all $ 1\le j \le m$,  where 
$F_j$ is a  division algebra over $\kk$ and $r_j  \in \NN$. Let $\Lambda \stackrel{\phi}\longhookrightarrow \Gamma$ be a nodal embedding of a basic semi-simple $\kk$-algebra $\Lambda$. Then $\phi$ is similar 
  to the  product of the following ``elementary'' nodal embeddings:
\begin{enumerate}
\item $F \stackrel{\mathsf{id}}\lar F$, where $F$ is a finite-dimensional  division algebra over $\kk$.
\item $F \xrightarrow{\mathsf{diag}_\tau} F \times F, \; a \mapsto (a, \tau(a))$, where $F$ is a finite-dimensional  division $\kk$-algebra and $\tau \in \mathsf{Aut}_{\kk}(F)$ ranges over representatives of 
the outer automorphism group $\Aut_{\kk}(F)/\mathsf{Inn}(F) := \mathsf{Out}_{\kk}(F)$, where $\mathsf{Inn}(F) \subseteq \Aut_{\kk}(F)$ is the subgroup of inner automorphisms of $F$. 
\item $K \stackrel{\imath}\lar L$, where  $K$ and $L$ are finite-dimensional  division $\kk$-algebras such that $\dim_{\kk}(L) = 2 \dim_{\kk}(K)$.
\item Canonical  embedding $F \times F \stackrel{\mathsf{can}}\lar M_2(F), (a, b) \mapsto 
\left(\begin{array}{cc}
a & 0 \\
0 & b \\
\end{array}
\right)$, where $F$ is a finite-dimensional  division $\kk$-algebra.
\item Regular embedding $L \stackrel{\mathsf{reg}}\lar M_2(K)$ associated with  an extension $K \stackrel{\imath}\lar L$, where $K$ and $L$ are finite-dimensional division $\kk$-algebras such that $\dim_{\kk}(L) = 2 \dim_{\kk}(K)$. 
\end{enumerate}
\end{theorem}

\begin{proof}  Let us  fix the following conventions and notation.
\begin{itemize}
\item If $F_i$ and $F_j$ are isomorphic $\kk$-algebras for some $1 \le i \ne j \le m$, then we view them as \emph{equal} and denote them by the same symbol. 
\item  For any $1 \le j \le m$, let  $S_j = F_j^{r_j} = \left(\begin{array}{c} F_j \\ \vdots 
\\ F_j \end{array}\right)$ be the simple left $\Gamma_j$-module (which is unique up to an isomorphism). 
\item We fix a decomposition $\Lambda = \Lambda_1 \times \dots \times \Lambda_n$, where $\Lambda_1, \dots, \Lambda_n$ are finite-dimensional division algebras over $\kk$. 
\item For any $1 \le i \le n$, let $e_i \in \Lambda$ be the primitive idempotent corresponding to the unit element of the factor $\Lambda_i$. Analogously, let $f_1, \dots, f_m \in \Gamma$ be the family of central orthogonal idempotents such that $\Gamma_j = f_j \Gamma = \Gamma f_j$ for all $1 \le j \le m$.  Of course, each $f_j$
corresponds to the unit element of the factor $\Gamma_j$.  
\item For any $1 \le i \le n$ and $1 \le j \le m$  consider the composition
$$
\sigma_{ji}: \Lambda_i \longhookrightarrow \Lambda \stackrel{\phi}\longhookrightarrow \Gamma \rightarrowdbl \Gamma_j.
$$
In a similar way, for any $1 \le i \ne i' \le n$ and $1 \le j \le m$ we have a map
$$
\sigma_{j(i i')}: \Lambda_i \times \Lambda_{i'}\longhookrightarrow 
\Lambda \stackrel{\phi}\longhookrightarrow \Gamma \rightarrowdbl \Gamma_j.
$$
\item Since $S_j$ is a finite length $\Lambda$-module, it admits a direct sum decomposition
$$
S_j \cong \bigoplus\limits_{i=1}^n \Lambda_i^{q_{ji}}
$$
for some uniquely determined multiplicities $q_{ji} \in \NN_0$, where $1\le i \le n$ and $1 \le j \le m$. It is clear that $q_{ji} > 0$ if and only if $\sigma_{ji} \ne 0$. 
\item For any $1 \le i \le n$ consider 
$
t_i := \sum\limits_{j = 1}^m r_j q_{ji}.
$
Obviously, $t_i$ is the multiplicity of  $\Lambda_i$ in $\Gamma$ viewed as a left $\Lambda$-module. 
\end{itemize} 

\smallskip
\noindent
\underline{Claim 1}. For any  $1 \le i \le n$ we have: $t_i \le 2$. 
As a consequence, 
\begin{equation}\label{E:Multiplicities}
\left|\{1 \le j \le m \, \big| q_{ji} \ge 1 \}\right| \le 2. 
\end{equation}
Moreover, $r_j \le 2$ for all $1 \le j \le m$. 

\smallskip
\noindent
\underline{Proof}.  As $\phi(\Lambda) \subseteq \Gamma$ is nodal and $\Lambda$ is basic, Theorem \ref{T:NodalCharact} implies that  $t_i \le 2$ for any $1 \le i \le n$.  Since for any  $1 \le j \le m$ there exists $1 \le i \le n$ such that $q_{ji} \ne 0$, we conclude that $r_j \le 2$, as asserted. \qed
\\[.5ex]
Note that the inequality (\ref{E:Multiplicities}) can be rephrased as follows: 
\begin{equation*}
\left|\{1 \le j \le m \, \big| \sigma _{ji} \ne 0 \}\right| \le 2. 
\end{equation*}
Summing up, each component $\Gamma_j$ is either a division algebra $F_j$ or the matrix algebra $M_2(F_j)$.  Moreover, each component $\Lambda_i$ of $\Lambda$ ``contributes'' to at most two different components $\Gamma_j$ and $\Gamma_{l}$ of $\Gamma$. \qed

\smallskip
\noindent
\underline{Claim 2}. For any  $1 \le j \le m$ we have:
\begin{equation}\label{E:MultiplicitiesDual}
\left|\{1 \le i \le n \, \big| q_{ji} \ge 1 \}\right| \le r_j. 
\end{equation}
Informally speaking, for each component $\Gamma_j$ there exist at most $r_j \le 2$ different components    of $\Lambda$ which ``contribute'' to $\Gamma_j$.

\smallskip
\noindent
\underline{Proof}. Given $1 \le j \le m$, assume that $q_{ji} \ne 0 \ne q_{jk}$ for some $1 \le i \ne k \le n$. Then $\phi(e_i) f_j$ and $\phi(e_k) f_j$ are two non-zero orthogonal idempotents of $\Gamma_j$. However, $\Gamma_j = M_{r_j}(F_j)$ with $r_j \le 2$  contains at most $r_j$ such idempotents, implying the inequality (\ref{E:MultiplicitiesDual}).  \qed

\smallskip
\noindent
\underline{Claim 3}. Assume that $1 \le j \le m$ and $1 \le i \ne k \le n$ are such that $q_{ji} \ne 0 \ne q_{jk}$. Then the following statements are true:
\begin{itemize}
\item $r_j = 2$.
\item $\Lambda_i \cong F_j \cong \Lambda_k$, $\sigma_{jl} = 0$ for all $l \notin\{i, k\}$ and $\Lambda_i \times \Lambda_k \stackrel{\sigma_{j(ik)}}\lar \Gamma_j = M_2(F_j)$ is similar  to the canonical  embedding. 
\end{itemize}

\smallskip
\noindent
\underline{Proof}. By Claim 1 we know that $r_j \le 2$. On the other hand, Claim 2 implies that $r_j \ge 2$. It follows that $r_j = 2$ and $\sigma_{jl} = 0$ for all $l \notin\{i, k\}$. It follows that 
$\Lambda_i \times \Lambda_k \stackrel{\sigma_{j(ik)}}\lar \Gamma_j$ is an injective  homomorphism of $\kk$-algebras. 

We have a direct sum decomposition $S_j \cong \Lambda_i^{p_i} \oplus \Lambda_k^{p_k}$ in the category of $\Lambda$-modules, where $p_i, p_k \ge 1$. 
Since $\Gamma_j = M_2(F_j)$,   the simple $\Gamma$-module $S_j$ has multiplicity two in the  decomposition of $\Gamma$ in a direct sum of simple  $\Gamma$-modules. 
Hence, $\Gamma$ viewed as a $\Lambda$-module, contains a direct summand $\Lambda_i^{2 p_i} \oplus \Lambda_k^{2p_k}$. On the other hand, according to 
Theorem \ref{T:NodalCharact},  we have a surjective homomorphism of $\Lambda$-modules $\Lambda^{{2}}  \twoheadarrow  \Gamma$. It follows that $p_i = 1 = p_k$  and  $S_j \cong \Lambda_i \oplus \Lambda_k$,  viewed as a $\Lambda$-module. 

Next, $\phi(e_i), \phi(e_k) \in \Gamma_j$ are primitive orthogonal idempotents and $f _j = \phi(e_i) + \phi(e_k)$. Since $\Gamma_j = M_2(F_j)$, there exists a unit $b \in \Gamma_j$ such that 
$$
b \phi(e_i) b^{-1} = u_{11} := \left(\begin{array}{cc}
1 & 0 \\
0 & 0
\end{array}
\right)
\quad \mbox{\rm and} \quad 
b \phi(e_k) b^{-1} = u_{22} := \left(\begin{array}{cc}
0 & 0 \\
0 & 1
\end{array}
\right).
$$
It follows that
$
\Lambda_i \cong u_{11} \Gamma_j u_{11} \cong F_j$ and $\Lambda_k \cong u_{22} \Gamma_j u_{22} \cong F_j.
$
Moreover, the following diagram
$$
\xymatrix{
\Lambda_i \times \Lambda_k \ar[r]^{\sigma_{j(ik)}} \ar[d]_-\cong & M_2(F_j) \ar[d]^-{\mathsf{Ad}_b} \\
F_j \times F_j \ar[r]^{\mathsf{can}} & M_2(F_j) 
}
$$
is commutative, proving the claim. \qed

\medskip
\noindent
Now we can classify all possible nodal embeddings
$
\prod\limits_{i = 1}^n \Lambda_i = \Lambda \stackrel{\phi}\lar \Gamma = \prod\limits_{j = 1}^m M_{r_j}(F_j),
$
where $\Lambda$ is assumed to be basic. 

\smallskip
\noindent
\underline{Case A}. Let $1 \le j \le m$  be such that $r_j = 1$. According to (\ref{E:MultiplicitiesDual}),  there exists a unique  $1 \le i \le n$ such that  $q_{ji} \ne 0$.

\smallskip
\noindent
\underline{Subcase A1}. Assume that $q_{ji} = 1$ and $q_{li} = 0$ for all $1 \le l \ne j \le m$. Then $\Lambda_i = \Gamma_j$ are some finite-dimensional division algebras over $\kk$ and $\sigma_{ji}$ is an isomorphism. Hence, we  get an elementary component of type (1).

\smallskip
\noindent
\underline{Subcase A2}. Assume that $q_{ji} = 1$ and $q_{li} \ne  0$
for some $1 \le l \ne j \le m$. Since $\sum\limits_{p = 1}^m r_p q_{pi} \le 2$, the only possibility is that $r_j = r_l = 1$ and $q_{ji} = 1 = q_{li}$, whereas  $q_{pi} = 0$ for all $p \notin \{j, l\}$. It follows that $\sigma _{pi} = 0$ for all $p \notin \{j, l\}$, whereas $\Lambda_i \stackrel{\sigma_{pi}}\lar \Gamma_p$ is an isomorphism of  $\kk$-algebras for $p \in \{j, l\}$. Let $\tau:= \sigma_{li} \sigma_{ji}^{-1}: F_j \lar F_l$. Identifying $\Lambda_i$, $\Gamma_j$ and $\Gamma_l$ with the same finite-dimensional division $\kk$-algebra $F$, we end up with an elementary nodal pair $F \xrightarrow{\mathsf{diag}_\tau} F \times F, \; a \mapsto (a, \tau(a))$ of type (2). Since we classify nodal pairs up to similarity, the similarity class of such an embedding depends only on the image of $\tau$ in the group $\mathsf{Out}_{\kk}(F)$.

\smallskip
\noindent
\underline{Subcase A3}. Assume that $q_{ji} = 2$. Since $t_i \le 2$, it follows that $\sigma_{li} = 0$ for all $1 \le l \ne j \le m$ (i.e.~the component $\Lambda_i$ contributes only to $\Gamma_j$).  It follows that $\Lambda_i \stackrel{\sigma_{ji}}\lar \Gamma_j = F_j$ is an injective homomorphism of $\kk$-algebras  and $\dim_{\kk}(\Gamma_j) = 2 \dim_{\kk}(\Lambda_i)$.  Hence, in this subcase we get an elementary component of type (3). 

\medskip
\noindent
\underline{Case B}. Let $1 \le j \le m$  be such that $r_j = 2$. 

\smallskip
\noindent
\underline{Subcase B1}. Assume that $1 \le i \ne k \le n$ are such that $q_{ji} \ne 0 \ne q_{jk}$. According to Claim 3, we have: $q_{ji} = 1 = q_{jk}$. Moreover, $\sigma_{jl} = 0$ for $l \notin \{i, k\}$ and
$\Lambda_i \times \Lambda_k \stackrel{\sigma_{j(ik)}}\lar \Gamma_j = M_2(F_j)$ is similar to an elementary component of type (4).

\smallskip
\noindent
\underline{Subcase B2}. Suppose that there exists exactly one $1 \le i \le n$ such that $q_{ji} \ne 0$. By Claim 1, we have: $t_i \le 2$. Since $r_j = 2$, we conclude that $q_{ji} = 1$ and $q_{li} = 0$ for all $1 \le l \ne j \le m$. Since $S_j \cong \Lambda_i$ as $\Lambda$-modules, we have: 
$$
\dim_{\kk}(\Lambda_i) = \dim_{\kk}(S_j) = r_j\,  \dim_{\kk}(F_j) = 
2\,  \mathsf{dim}_{\kk}(F_j).
$$
Therefore,
$$
\dim_{\kk}(\Gamma_j) = r_j^2 \, \dim_{\kk}(F_j) = 4\,  \dim_{\kk}(F_j) = 2\, \dim_{\kk}(\Lambda_i)
$$
which shows  that we are precisely in the setting of Corollary \ref{C:Regular}. It follows that $\Lambda_i \stackrel{\sigma_{ji}}\lar \Gamma_j = M_2(F_j)$ is an elementary nodal pair of type (5).

\smallskip
\noindent
Since we exhausted all the cases, the theorem is proven. 
\end{proof}

\begin{corollary}\label{C:Real} We illustrate Theorem \ref{T:SemiSimpleNodal} by describing the similarity types of  elementary nodal embeddings for $\kk = \RR$. 
\begin{enumerate}
\item[(a)] Type  (1) embeddings are $\FF \stackrel{\mathsf{id}}\lar \FF$, where  $\FF \in \bigl\{\RR, \CC, \HH\bigr\}$. 
\item[(b)] Type (2) embeddings are
$\FF \stackrel{\mathsf{diag}}\lar \FF \times \FF$ for $\FF \in \bigl\{\RR, \CC, \HH\bigr\}$. Additionally, we have the embedding 
$\CC \stackrel{\mathsf{diag}^\ast}\lar \CC \times \CC, z \mapsto (z, \bar{z})$. 
\item[(c)] Type (3) embeddings are  $\RR \subset \CC$ and $\CC \subset \HH$. 
\item[(d)] Type (4) embeddings are $\FF \times \FF \stackrel{\mathsf{can}}\lar  M_2(\FF)$ for each  $\FF \in  \bigl\{\RR, \CC, \HH\bigr\}$.  
\item[(e)] Finally,  type (5) embeddings are regular embeddings given by  (\ref{E:regular1}) and 
(\ref{E:regular2}).
\end{enumerate}

\end{corollary}

\begin{remark}
In \cite{Iyama}, Iyama introduced the class of the so-called \emph{quadratic orders} and classified the corresponding indecomposable finite-length modules. In fact, quadratic orders form a large subclass of nodal orders. The list of the so-called \emph{minimal quadratic extensions} (see \cite[Section 1.2]{Iyama}) essentially coincides with the list of elementary embeddings (1)--(5) from Theorem \ref{T:SemiSimpleNodal}. The only difference is that the embeddings $F \xrightarrow{\mathsf{diag}_\tau} F \times F$ with $[\tau] \ne [\mathsf{id}] \in \mathsf{Out}_{\kk}(F)$ do not appear  in \cite{Iyama}. The main advantage of the definition of nodal orders given in \cite{NodalFirst, Nodal} and used in this work is that it is purely intrinsic. 
\end{remark}

\section{Classification of real nodal orders}\label{S:RealNodal}

\noindent
In this section, we give a classification of real nodal orders. 
\begin{enumerate}
\item[(i)] Let $\Pi = \left\{\mathsf{re}, \mathsf{co}, \mathsf{qt}\right\}$.
 We put: $\FF_{\mathsf{re}} = \RR$, $\FF_{\mathsf{co}} = \CC$ and $\FF_{\mathsf{qt}} = \HH$, which defines a bijection between 
$\Pi$ and the  set of isomorphism classes of finite-dimensional division algebras  over $\RR$. 
\item[(ii)]  Similarly, let $\Xi = \left\{\mathsf{re}, \mathsf{cx}, \mathsf{tc}, \mathsf{qt}\right\}$. We put 
$\OO_{\mathsf{re}} = \RR\llbracket t\rrbracket$, $\OO_{\mathsf{cx}} = \CC\llbracket t\rrbracket$, $\OO_{\mathsf{tc}} = \CC\llbracket t\rrbracket^{\mathsf{tw}}$ and $\OO_{\mathsf{qt}} = \HH\llbracket t\rrbracket$.
This correspondence defines  a bijection between  $\Xi$ and the set of isomorphism classes of  maximal  local real orders. 
\item[(iii)]
We fix a surjective  map $\Xi \rightarrowdbl \Pi$, which sends $\mathsf{cx}$ and $\mathsf{tc}$ to $\mathsf{co}$ and $\mathsf{re}$ and  $\mathsf{qt}$ to themselves. This map assigns to a maximal real order its residue skew field. 
\end{enumerate}

\begin{definition}\label{D:ParameterSet}
Consider tuples  $\bigl(\Omega, \prec,  \tau, \sim,  \alpha, \beta, \gamma\bigr)$ defined as follows. 

\begin{enumerate}
\item[(a)] $(\Omega, \prec)$ is a finite partially ordered set, 
which is a union of chains: $\Omega = \bigsqcup\limits_{i \in I} \Omega_i$.
\item[(b)] $I \stackrel{\tau}\lar \Xi$ is any map. It defines  an induced map $I \stackrel{\chi}\lar \Pi$ given  as the composition
$$
\xymatrix{
& I \ar[ld]_{\tau} \ar[rd]^\chi & \\
\Xi \ar@{->>}[rr] & & \Pi
}
$$  
 Abusing the notation, we obtain  further   induced maps $\Omega \stackrel{\tau}\lar \Xi$ and $\Omega \stackrel{\chi}\lar \Pi$, defined by the formulae $\tau(\omega) = \tau(i)$ and $\chi(\omega) = \chi(i)$ for any $\omega \in \Omega_i$. 
 We have decompositions 
 $$
 I = I_{\mathsf{re}} \sqcup I_{\mathsf{co}} \sqcup I_{\mathsf{qt}} \quad\mbox{\rm and} \quad I_{\mathsf{co}} = I_{\mathsf{cx}} \sqcup I_{\mathsf{tc}}
 $$
 as well as 
 $$
 \Omega = \Omega_{\mathsf{re}} \sqcup \Omega_{\mathsf{co}} \sqcup \Omega_{\mathsf{qt}} \quad\mbox{\rm and} \quad \Omega_{\mathsf{co}} = \Omega_{\mathsf{cx}} \sqcup \Omega_{\mathsf{tc}}
 $$
 defined in an obvious way. 
\item[(c)] $\sim$ is a symmetric binary  relation on $\Omega$ satisfying the following conditions:
\begin{enumerate}
\item[(i)] For any $\omega \in \Omega$ there exists at most one $\widetilde\omega \in \Omega$ such that $\omega \sim \widetilde\omega$. 
\item[(ii)] If $\omega' \ne \omega''$ and $\omega' \sim \omega''$ then $\chi(\omega') = \chi(\omega'')$. 
\end{enumerate}
We put: $\Omega^{\mathsf{s}} = \bigl\{\omega \in \Omega \, \big| \, \nexists \, \widetilde\omega \, \in \Omega: \omega \sim \widetilde\omega \bigr\}$, 
$\Omega^{\mathsf{d}} = \bigl\{\omega \in \Omega \, \big| \, \omega \sim \omega \bigr\}$ and $\Omega^{\mathsf{g}} = \bigl\{\omega \in \Omega \, \big| \, \exists \, \widetilde\omega \, \in \Omega: \omega \sim \widetilde\omega \; \mbox{and} \; \omega \ne \widetilde\omega\bigr\}$ (the corresponding abbreviations stand  for  \emph{single}, \emph{doubled} and \emph{glued} elements of $\Omega$, respectively). Let $\overline{\Omega}_{\mathsf{co}}^{\mathsf{g}} = {\Omega}_{\mathsf{co}}^{\mathsf{g}}/\sim$ be the set of equivalence  classes of elements of ${\Omega}_{\mathsf{co}}^{\mathsf{g}}$ with respect to the binary relation $\sim$. 
\item[(d)] Finally, we have the following three maps $\alpha, \beta$ and $\gamma$: 
\begin{enumerate}
\item[(i)] $\Omega^{\mathsf{s}} \stackrel{\alpha}\lar \bigl\{\mathsf{id}, \mathsf{ex} \bigr\}$ is any map such that $\alpha(\omega) = \mathsf{id}$ if $\chi(\omega) = \mathsf{re}$. 
\item[(ii)] $\Omega^{\mathsf{d}} \stackrel{\beta}\lar \{\mathsf{can}, \mathsf{reg}\}$ is any map such that
$\beta(\omega) = \mathsf{can}$ if  $\chi(\omega) = \mathsf{qt}$.
\item[(iii)] $\overline{\Omega}_{\mathsf{co}}^{\mathsf{g}} \stackrel{\gamma}\lar \bigl\{+1, -1 \bigr\}$ is any map. 
\end{enumerate}
\end{enumerate}
\end{definition}

\smallskip
\noindent
\textbf{Normal form of a basic real nodal order}.
Let us  now describe the procedure which associates to a datum $\bigl(\Omega, \prec, \tau, \sim,  \alpha, \beta, \gamma\bigr)$ as above  the corresponding \emph{basic} real nodal order $A = A\bigl(\Omega, \prec, \tau, \sim,  \alpha, \beta, \gamma\bigr)$. In the notation of  Theorem \ref{T:Cartesian}, we  define $A$ as $H \vee \Lambda$, where $H$ is a real hereditary order, $\Lambda$ is a basic finite-dimensional semi-simple algebra over $\RR$ and $\Lambda \stackrel{\jmath}\lar \bar{H}$ a nodal embedding defined as follows. 

\smallskip
\noindent
A)  Description of the hereditary order $H = H(\Omega, \prec, \tau, \sim)$.
\begin{enumerate}
\item[(i)] The relation $\sim$ defines  the following  function:
\begin{equation}\label{E:WeightHeredOrder}
\Omega \stackrel{\upsilon}\lar \left\{1, 2\right\}, \;  \upsilon(\omega)= \left\{
\begin{array}{ll}
1 & \mbox{\rm if} \;  \omega \in \Omega^{\mathsf{s}} \sqcup \Omega^{\mathsf{g}} \\
2 & \mbox{\rm if} \;  \omega \in \Omega^{\mathsf{d}}. 
\end{array}
\right.
\end{equation}
\item[(ii)] For any $k \in I$ 
let $\Omega_k = \bigl\{\omega_1^{(k)}, \omega_2^{(k)}, \dots,  \omega_{l_k}^{(k)} \bigr\}$ with the  ordering given  by 
$\omega_1^{(k)} \prec \omega_2^{(k)} \prec \dots \prec  \omega_{l_k}^{(k)}
$. Let $\vec{d}_k = (d_1, \dots, d_{l_k})$, where 
$d_i = \upsilon(\omega^{(k)}_i)$ for each $1 \le i \le l_k$. Then we put:
\begin{equation}\label{E:HeredCover}
H = \prod\limits_{k \in I} H_k =  \prod\limits_{k \in I} H\bigl(\OO_{\tau(k)}, \vec{d}_k\bigr).
\end{equation}
In these terms, we  get a natural isomorphism
$$
\bar{H} \cong \prod\limits_{\omega \in \Omega} \bar{H}_{\omega} =  \prod\limits_{\omega \in \Omega} M_{\upsilon(\omega)}\bigl(\FF_{\chi(\omega)}\bigr).
$$
\end{enumerate}
\smallskip
\noindent
B)  Description of the basic semi-simple $\RR$-algebra $\Lambda = \Lambda\bigl(\Omega, \chi, \sim,  \alpha, \beta, \gamma\bigr)$ as well as the nodal embedding $\Lambda \stackrel{\jmath}\lar \bar{H}$.

\begin{enumerate}
\item[(i)] Let $\overline{\Omega}^{\mathsf{g}} = {\Omega}^{\mathsf{g}}/\sim$ be the set of equivalence classes of elements of ${\Omega}^{\mathsf{g}}$ with respect to  the  relation $\sim$. Then we have a canonical surjective map $\Omega^{\mathsf{g}} \stackrel{\pi_{\mathsf{g}}}\rightarrowdbl \overline{\Omega}^{\mathsf{g}}$. Next, starting with the set $\Omega^{\mathsf{d}}$, we define a new set $\overline{\Omega}^{\mathsf{d}}$ by replacing each element $\omega \in 
\Omega^{\mathsf{d}}$ such that $\beta(\omega) = \mathsf{can}$ by two new elements $\omega_+$ and 
$\omega_-$. In this way, we get a surjective map $\overline{\Omega}^{\mathsf{d}} \stackrel{\pi_{\mathsf{d}}}\rightarrowdbl {\Omega}^{\mathsf{d}}$. We put:
\begin{equation}\label{E:OmegaBar}
\overline{\Omega} = \Omega^{\mathsf{s}} \sqcup \overline{\Omega}^{\mathsf{d}} \sqcup \overline{\Omega}^{\mathsf{g}}.
\end{equation}
The primitive idempotents of the algebra $\Lambda$ constructed below are in bijection with elements of the set $\overline{\Omega}$. More precisely, we put  $\Lambda = 
\prod\limits_{\xi \in \overline{\Omega}} \Lambda_{\xi} = \prod\limits_{\xi \in \overline{\Omega}} \FF_{\chi(\xi)}$.
\item[(ii)] Let $\xi \in \Omega^{\mathsf{s}}$. Then we have a nodal embedding $\Lambda_\xi \stackrel{\jmath_{\xi}}\lar \bar{H}_\xi$ defined as follows. 
\begin{enumerate}
\item[(a)] If $\alpha(\xi) = \mathsf{id}$ then this embedding is
$$
\Lambda_\xi = \FF_{\chi(\xi)} \stackrel{\mathsf{id}}\lar \FF_{\chi(\xi)} = \bar{H}_\xi. 
$$
\item[(b)] If $\alpha(\xi) = \mathsf{ex}$ then this embedding is 
$$
\Lambda_\xi = \RR  \longhookrightarrow \CC = \bar{H}_\xi \quad  \mbox{\rm if} \quad \chi(\xi) = \mathsf{co}
$$
and 
$$
\Lambda_\xi = \CC  \longhookrightarrow \HH = \bar{H}_\xi \quad \mbox{\rm if} \quad  \chi(\xi) = \mathsf{qt}.
$$
\end{enumerate}
\item[(iii)] Let $\xi \in \overline{\Omega}^{\mathsf{g}}$ and 
$\left\{\omega', \omega'' \right\} = \pi^{-1}_{\mathsf{g}}(\xi) \subset \Omega^{\mathsf{g}}$. 
Then we have a nodal embedding $\Lambda_\xi = \FF_{\chi(\xi)}\stackrel{\jmath_{\xi}}\lar \FF_{\chi(\xi)} \times \FF_{\chi(\xi)} = \bar{H}_{\omega'} \times \bar{H}_{\omega''}$  defined as follows: 
$$
\jmath_{\xi} = 
\left\{
\begin{array}{ll}
\mathsf{diag}^\ast & \mbox{\rm if} \; \chi(\xi) = \mathsf{co} \; \mbox{\rm and} \; \gamma(\xi) = -1 \\
\mathsf{diag} & \mbox{\rm in all the other cases}.
\end{array}
\right.
$$
\item[(iv)] Finally, for any $\omega \in \Omega^{\mathsf{d}}$ we define the corresponding nodal embedding $\jmath_\omega$ as follows. 
\begin{enumerate}
\item[(a)] If $\pi^{-1}_{\mathsf{d}}(\omega) = \left\{\omega_+, \omega_-\right\}$ then $\jmath_\omega$ is
$$
\Lambda_{\omega_+} \times \Lambda_{\omega_-} = 
\FF_{\chi(\omega)} \times \FF_{\chi(\omega)} \stackrel{\mathsf{can}}\lar  M_2\bigl(\FF_{\chi(\omega)}\bigr) = \bar{H}_{\omega}.
$$
\item[(b)] If $\pi^{-1}_{\mathsf{d}}(\omega) = \{\omega\}$ then $\jmath_\omega$ is 
$$
\Lambda_\omega = \CC  \stackrel{\mathsf{reg}}\longhookrightarrow M_2(\RR)  = \bar{H}_\omega \quad  \mbox{\rm if} \quad \chi(\omega) = \mathsf{re}
$$
and 
$$
\Lambda_\omega = \HH  \stackrel{\mathsf{reg}}\longhookrightarrow M_2(\CC)  = \bar{H}_\omega \quad  \mbox{\rm if} \quad \chi(\omega) = \mathsf{co}.
$$
\end{enumerate}
\end{enumerate}

\begin{definition}\label{D:TuplesEquivalentC}
Let  $\bigl(\Omega, \prec,  \tau, \sim,  \alpha, \beta, \gamma\bigr)$ and $\bigl(\Omega', \prec',  \tau', \sim',  \alpha', \beta', \gamma'\bigr)$ 
be a pair of tuples from Definition \ref{D:ParameterSet}. 
They are called \emph{equivalent} if and only if there exist a bijection $\Omega \stackrel{\varphi}\lar \Omega'$ and a function $I_{\mathsf{co}} \stackrel{\eta}\lar \{+1, -1\}$ satisfying the following properties.

\begin{enumerate}
\item[(a)]  If $\omega_1, \omega_2 \in \Omega$ belong to the same chain then  $\varphi(\omega_1), \varphi(\omega_2) \in \Omega'$  belong to the same chain, too. Hence, $\varphi$ induces a bijection $I \stackrel{\phi}\lar I'$ of the sets of chains of $\Omega$ and $\Omega'$. 
Moreover, for any $i \in I$ the restricted bijection $\Omega_i \stackrel{\varphi}\lar \Omega'_{\phi(i)}$ is a cyclic permutation with respect to $\prec$ and $\prec'$. 
\item[(b)] The diagram 
$$
\xymatrix{
I \ar[rr]^-\phi \ar[rd]_-{\tau} & & I' \ar[ld]^-{\tau'}\\
& \Xi & 
}
$$
is commutative. 
\item[(c)] For any $\omega_1, \omega_2 \in \Omega$ we have: $\omega_1 \sim \omega_2$ if and only if $\varphi(\omega_1) \sim' \varphi(\omega_2)$. Similarly, the diagrams 
\vspace{-3mm}
$$
\begin{array}{ccc}
\xymatrix{
\Omega^{\mathsf{s}} \ar[rr]^-\varphi \ar[rd]_-{\alpha} & & \Omega'^{\mathsf{s}}  \ar[ld]^-{\alpha'}\\
& \{\mathsf{id}, \mathsf{ex}\} & 
}
& \begin{array}{c} \\ \\ \\ \mbox{\rm and}\end{array} & 
\xymatrix{
\Omega^{\mathsf{d}} \ar[rr]^-\varphi \ar[rd]_-{\beta} & & \Omega'^{\mathsf{d}}  \ar[ld]^-{\beta'}\\
& \{\mathsf{can}, \mathsf{reg}\} & 
}
\end{array}
$$
are commutative. 
\item[(d)] Consider the function $\Omega_{\mathsf{co}} \stackrel{\delta_\varphi}\lar \{+1, -1\}$ defined as follows. 
\begin{enumerate}
\item[(i)] If $\omega \in \Omega_{\mathsf{cx}}$ then $\delta_\varphi(\omega) = 1$. 
\item[(ii)] Assume now that $\omega \in \Omega_{\mathsf{tc}} \cap \Omega_i$ for some $i \in I$. 
Let $\omega_\circ$ be the minimal element of $\Omega_i$. Then we put:
\begin{equation}
\delta_\varphi(\omega) = \left\{
\begin{array}{cl}
-1 & \mbox{\rm if} \; \varphi(\omega) \prec \varphi(\omega_\circ) \quad \mbox{\rm in} \quad  \Omega'_{\phi(i)} \\
+1 & \mbox{\rm if} \; \varphi(\omega) \succeq \varphi(\omega_\circ) \quad  \mbox{\rm in} \quad  \Omega'_{\phi(i)}.
\end{array}
\right.
\end{equation}
\item[(iii)] Let $\Omega_{\mathsf{co}} \stackrel{\widetilde\eta}\lar \bigl\{+1, -1\bigr\}$ be the composition of $\eta$ with the projection $\Omega_{\mathsf{co}} \rightarrowdbl I_{\mathsf{co}}$. 
Then we require that 
$$
\gamma'\bigl(\overline{\{\varphi(\omega_1), \varphi(\omega_2)\}}\bigr) = \delta_\varphi(\omega_1) \delta_\varphi(\omega_2) \widetilde\eta(\omega_1) \widetilde\eta(\omega_2) \gamma\bigl(\overline{\{\omega_1, \omega_2\}}\bigr)
$$
for any $\omega_1 \ne  \omega_2 \in \Omega_{\mathsf{co}}$ such that $\omega_1 \sim \omega_2$. 
\end{enumerate}
\end{enumerate}
\end{definition}

\begin{theorem}\label{T:NodalClassification} The isomorphism classes of basic real nodal orders are parametrized by  the equivalence types of tuples $\bigl(\Omega, \prec,  \tau, \sim,  \alpha, \beta, \gamma\bigr)$  from Definition \ref{D:ParameterSet}.
\end{theorem}

\begin{proof} Let $A$ be a basic real nodal order. By  Theorem \ref{T:Cartesian} we have:
$A \cong H \vee \Lambda$, where $H$ is a real hereditary order, $\Lambda$ is a basic finite-dimensional semi-simple algebra over $\RR$ and $\Lambda \stackrel{\jmath}\lar \bar{H}$ is a nodal embedding. First note that $\jmath$ is a product of elementary nodal embeddings described in Corollary \ref{C:Real}. 

\smallskip
\noindent
By Theorem \ref{T:StandardOrder} and Theorem \ref{T:HeredOrders}, any  real hereditary order $H$ is isomorphic to
$$
H = H(O_1, \vec{d}_1) \times \dots \times H(O_r, \vec{d}_r)
$$
for some $r \in \NN$, where $O_i \in \bigl\{\OO_{\mathsf{re}}, \OO_{\mathsf{cx}}, \OO_{\mathsf{tc}}, \OO_{\mathsf{qt}} \bigr\}$ for all $1  \le i \le r$; see Summary \ref{SummaryRealHereditary}. 

Now we explain how the corresponding tuple $\bigl(\Omega, \prec,  \tau, \sim,  \alpha, \beta, \gamma\bigr)$ is  constructed. 
Because of the shape of the embedding $\Lambda \stackrel{\jmath}\lar \bar{H}$ we can conclude that entries of each $\vec{d}_i$ are $1$ or $2$. If  $\vec{d}_i \in \NN^{l_i}$ then we take a totally ordered chain $\Omega_i$ with $l_i$ elements and put 
$\Omega = \Omega_1 \sqcup \dots \sqcup \Omega_r$. This defines a partially ordered set 
$(\Omega, \prec)$.  In these terms, $I = \left\{1, \dots, r\right\}$. We define $I \stackrel{\tau}\lar \Xi$ by the rule $O_i \cong \OO_{\tau(i)}$ for each $i \in I$. 
The remaining data $(\sim, \alpha, \beta, \gamma)$ can be extracted directly from the structure of the embedding $\Lambda \stackrel{\jmath}\lar \bar{H}$ by reversing Part B of the construction of the normal form of a basic real nodal order; see Corollary \ref{C:Real}. 
It follows  that $A \cong A\bigl(\Omega, \prec,  \tau, \sim,  \alpha, \beta, \gamma\bigr)$, as asserted. 

It remains to be clarified when two  tuples  $\bigl(\Omega, \prec,  \tau, \sim,  \alpha, \beta, \gamma\bigr)$ and $\bigl(\Omega', \prec',  \tau', \sim',  \alpha', \beta', \gamma'\bigr)$ define isomorphic  real nodal orders.
Let $A = H  \vee \Lambda$ and $A' = H' \vee \Lambda'$. By Theorem \ref{T:Cartesian}, any 
isomorphism $A \lar A'$ can be extended to an isomorphism of the corresponding hereditary hulls $$H(\Omega, \prec, \tau, \sim) = H \stackrel{f}\lar H' = H(\Omega', \prec', \tau', \sim').$$ 
A classification of all such isomorphims is given by Corollary \ref{C:SummaryAutomorphisms}. Since $f$ identifies components of $H$ and $H'$ with respect to the decompositions (\ref{E:DecompHered}), 
it defines  bijections $\Omega \stackrel{\varphi}\lar \Omega'$ and 
$I \stackrel{\phi}\lar I'$ such that
$$
\xymatrix{
\Omega \ar[r]^{\varphi} \ar@{->>}[d] & \Omega' \ar@{->>}[d]\\
I \ar[r]^{\phi}  & I'
}
$$
is commutative. 
Moreover, $\tau' \phi = \tau$.

Let $\bar{H} \stackrel{\bar{f}}\lar \bar{H}'$ be the isomorphism induced by $f$. Then there exists a unique isomorphism of $\RR$-algebras $\Lambda \stackrel{\tilde{f}}\lar \Lambda'$ such that the diagram 
\begin{equation}\label{E:Compatibilty}
\begin{array}{c}
\xymatrix{
\Lambda \ar[r]^-{\tilde{f}} \ar@{_{(}->}[d]_{\jmath} & \Lambda' \ar@{^{(}->}[d]^-{\jmath'}\\
\bar{H} \ar[r]^-{\bar{f}}  & \bar{H}'
}
\end{array}
\end{equation}
is commutative. Existence of such $\tilde{f}$ implies  that 
there exists a function $I_{\mathsf{co}} \stackrel{\eta}\lar \{+1, -1\}$ such that 
$(\varphi, \eta)$ defines  equivalence  of tuples $\bigl(\Omega, \prec,  \tau, \sim,  \alpha, \beta, \gamma\bigr)$ and $\bigl(\Omega', \prec',  \tau', \sim',  \alpha', \beta', \gamma'\bigr)$; see (\ref{E:type3}) and (\ref{E:type3tc}) which explain the origin of the function $\eta$.

Conversely, any equivalence of tuples 
$$
(\varphi, \eta): \bigl(\Omega, \prec,  \tau, \sim,  \alpha, \beta, \gamma\bigr)\lar \bigl(\Omega', \prec',  \tau', \sim',  \alpha', \beta', \gamma'\bigr)
$$
defines an isomorphism  $H \stackrel{f}\lar H'$, for which there exists an isomorphism 
$\Lambda \stackrel{\tilde{f}}\lar \Lambda'$ making the  diagram (\ref{E:Compatibilty}) commutative. Hence, we get an isomorphism of nodal orders 
$$
A\bigl(\Omega, \prec,  \tau, \sim,  \alpha, \beta, \gamma\bigr)\lar A\bigl(\Omega', \prec',  \tau', \sim',  \alpha', \beta', \gamma'\bigr),
$$
as asserted. 
\end{proof}

\begin{example}\label{E:RealNodalExamplesComb} 
We illustrate Theorem \ref{T:NodalClassification} by describing the real nodal orders from Example \ref{E:RealNodalExamples} in terms of the combinatorial data of Definition \ref{D:ParameterSet}, keeping the same numbering as in Example \ref{E:RealNodalExamples}.

\begin{enumerate}
\item[(i)] Here we have $I = \{\ast\}$, $\Omega = \Omega_\ast = \{\omega\}$ and $\omega \not\sim \omega$. In particular, $\Omega = \Omega^{\mathsf{s}}$. Moreover, $\tau(\ast) = \mathsf{cx}$ and $\alpha(\omega) = \mathsf{ex}$.
\item[(ii)] Let $I = \bigl\{\blacktriangle, \vartriangle\bigr\}$, whereas $\tau(\blacktriangle) = \mathsf{tc}$ and $\tau(\vartriangle) = \mathsf{cx}$.  Then we have $\Omega = \Omega_\blacktriangle \sqcup \Omega_\vartriangle$, $\Omega_\blacktriangle = \left\{\omega_\blacktriangle\right\}$, $\Omega_\vartriangle = \left\{\omega_\vartriangle\right\}$ and
$\omega_\blacktriangle \sim \omega_\vartriangle$. 
Next, 
$\Omega = \Omega^{\mathsf{g}}_{\mathsf{co}}$ and $\gamma(\omega) = 1$ for the unique element $\omega \in \overline{\Omega}^{\mathsf{g}}_{\mathsf{co}}$.
\item[(iii)] Here we have $I = \{\ast\}$, $\Omega = \Omega_\ast = \{\omega\}$ and $\omega \sim \omega$. In particular, $\Omega = \Omega^{\mathsf{d}}$. Moreover, $\tau(\ast) = \mathsf{re}$ and $\beta(\omega) = \mathsf{reg}$.
\item[(iv)] Let $I = \{\ast\}$, $\tau(\ast) = \mathsf{cx}$  and $\Omega = \Omega_\ast = \{\omega_1, \omega_2\}$ with $\omega_1 \prec \omega_2$ and $\omega_1 \sim \omega_2$. In particular, $\Omega = \Omega^{\mathsf{g}}_{\mathsf{co}}$ and $\gamma(\omega) = -1$ for the unique element $\omega \in \overline{\Omega}^{\mathsf{g}}_{\mathsf{co}}$.
\item[(v)] Here we have $I = \{\ast\}$, $\tau(\ast) = \mathsf{cx}$,  $\Omega = \Omega_\ast = \{\omega\}$ and $\omega \sim \omega$. Next, $\Omega = \Omega^{\mathsf{d}}$ and $\beta(\omega) = \mathsf{reg}$.
\item[(vi)] Here we have $I = \{\ast\}$, $\tau(\ast) = \mathsf{re}$ and $\Omega = \Omega_\ast = \{\omega_1, \omega_2\}$ with $\omega_1 \prec \omega_2$, $\omega_1 \sim \omega_1$ and $\omega_2$ being untied. Then $\Omega = \Omega^{\mathsf{d}} \sqcup
\Omega^{\mathsf{s}}$, where $\Omega^{\mathsf{d}} = \{\omega_1\}$ and $\Omega^{\mathsf{s}} = \{\omega_2\}$, whereas
$\alpha(\omega_2) = \mathsf{id}$ and $\beta(\omega_1) = \mathsf{reg}$.
\item[(vii)] In this case, $I = \{\ast\}$, $\tau(\ast) = \mathsf{qt}$, $\Omega = \Omega_\ast = \{\omega\}$ and $\omega$ is untied, i.e.~$\Omega = \Omega^{\mathsf{s}}$. We have: $\alpha(\omega) = \mathsf{ex}$. 
\end{enumerate}
\end{example}

\begin{example} For completeness, we also provide combinatorial data for nodal orders from Example \ref{Ex:NodalOrdersCx} for $\kk = \RR$, keeping the same numbering.
\begin{enumerate}
\item[(i)] Let $I = \bigl\{\blacktriangle, \vartriangle\bigr\}$ and $\tau(\blacktriangle) = 
\mathsf{re} = \tau(\vartriangle)$.  Then we have $\Omega = \Omega_\blacktriangle \sqcup \Omega_\vartriangle$, $\Omega_\blacktriangle = \left\{\omega_\blacktriangle\right\}$, $\Omega_\vartriangle = \left\{\omega_\vartriangle\right\}$ and
$\omega_\blacktriangle \sim \omega_\vartriangle$.
\item[(ii)] Here we have  $I = \{\ast\}$, $\tau(\ast) = \mathsf{re}$  and $\Omega = \Omega_\ast = \{\omega_1, \omega_2\}$ with $\omega_1 \prec \omega_2$ and $\omega_1 \sim \omega_2$.
\item[(iii)] In this case, $I = \{\ast\}$, $\tau(\ast) = \mathsf{re}$,  $\Omega = \Omega_\ast = \{\omega\}$ and $\omega \sim \omega$. Next, $\Omega = \Omega^{\mathsf{d}}$ and $\beta(\omega) = \mathsf{can}$.
\item[(iv)] Here we have $I = \{\ast\}$, $\tau(\ast) = \mathsf{re}$ and $\Omega = \Omega_\ast = \{\omega_1, \omega_2\}$ with $\omega_1 \prec \omega_2$, $\omega_1 \sim \omega_1$ and $\omega_2$ being untied. Then $\Omega = \Omega^{\mathsf{d}} \sqcup
\Omega^{\mathsf{s}}$, where $\Omega^{\mathsf{d}} = \{\omega_1\}$ and $\Omega^{\mathsf{s}} = \{\omega_2\}$, whereas
$\alpha(\omega_2) = \mathsf{id}$ and $\beta(\omega_1) = \mathsf{can}$.
\end{enumerate}
\end{example}

\medskip
\noindent
\textbf{Normal form of arbitrary real nodal orders}. Let $(\Omega, \prec,  \tau, \sim,  \alpha, \beta, \gamma\bigr)$ be a tuple as in Definition \ref{D:ParameterSet} and 
$\overline{\Omega} \stackrel{\mathsf{wt}}\lar \NN$ be any function, where the set $\overline\Omega = \Omega^{\mathsf{s}} \sqcup \overline{\Omega}^{\mathsf{d}} \sqcup \overline{\Omega}^{\mathsf{g}}$ is defined 
by (\ref{E:OmegaBar}).  Then the isomorphism classes of arbitrary real nodal orders are parametrized by extended tuples $(\Omega, \prec,  \tau, \sim,  \alpha, \beta, \gamma, \mathsf{wt}\bigr)$.

First, we define the function $\Omega \stackrel{\mathsf{wt}_\upsilon}\lar \NN$, which generalizes the function  
$\upsilon$ given by (\ref{E:WeightHeredOrder}),
by the following rules.
\begin{enumerate}
\item[(a)] If $\omega \in \Omega^{\mathsf{s}}$ then $\mathsf{wt}_\upsilon(\omega) = 
\mathsf{wt}(\omega)$. 
\item[(b)] If $\omega \in \Omega^{\mathsf{g}}$ then $\mathsf{wt}_\upsilon(\omega) = 
 \mathsf{wt}\bigl(\pi_{\mathsf{g}}(\omega)\bigr)$, where $\Omega^{\mathsf{g}} \stackrel{\pi_{\mathsf{g}}}\rightarrowdbl \overline{\Omega}^{\mathsf{g}}$ is the canonical surjection.
\item[(c)] Let $\omega \in \Omega^{\mathsf{d}}$. Recall that we have a surjective map 
$\overline{\Omega}^{\mathsf{d}} \stackrel{\pi_{\mathsf{d}}}\rightarrowdbl {\Omega}^{\mathsf{d}}$. If $\pi_{\mathsf{d}}^{-1}(\omega) = \{\omega\}$ then $\mathsf{wt}_\upsilon(\omega) = 
2 \mathsf{wt}(\omega)$. If $\pi_{\mathsf{d}}^{-1}(\omega) = \{\omega_+, \omega_-\}$ then $\mathsf{wt}_\upsilon(\omega) = 
\mathsf{wt}(\omega_+) + \mathsf{wt}(\omega_-)$.
\end{enumerate}
In fact, $\mathsf{wt}_\upsilon = \upsilon$ if $\mathsf{wt}(\bar\omega) = 1$ for all $\bar\omega \in \overline\Omega$.

We define the hereditary order $H = H(\Omega, \prec, \tau, \sim, \mathsf{wt})$ by the same formula (\ref{E:HeredCover}). However, now for any  $k \in I$ the vector  $\vec{d}_k$ is defined by replacing the function $\upsilon$ via  its weighted version $\mathsf{wt}_\upsilon$. Next, we put   
$$\Lambda = \Lambda(\overline{\Omega}, \chi, \mathsf{wt}) = \prod\limits_{\xi \in \overline{\Omega}} \Lambda_{\xi} =
\prod\limits_{\xi \in \overline{\Omega}} M_{\mathsf{wt}(\xi)}\bigl(\FF_{\chi(\xi)}\bigr).$$
The embedding $\Lambda \stackrel{\jmath}\lar \bar{H}$ is defined analogously to the basic case.  For example, let  $\omega \in \Omega^{\mathsf{d}}$ be such that $\pi_{\mathsf{d}}^{-1}(\omega) = \{\omega_+, \omega_-\}$. Then we have an embedding $\Lambda_{\omega_+} \times \Lambda_{\omega_-} \stackrel{\jmath_\omega}\lar \bar{H}_{\omega} $ defined as follows:
\begin{equation}
M_{\mathsf{wt}(\omega_+)}\bigl(\FF_{\chi(\omega)}\bigr) \times 
M_{\mathsf{wt}(\omega_-)}\bigl(\FF_{\chi(\omega)}\bigr)
\lar  M_{\mathsf{wt}_{\upsilon}(\omega)}\bigl(\FF_{\chi(\omega)}\bigr), \quad (x_+, x_-) \mapsto
\left(
\begin{array}{cc}
x_+ & 0 \\
0 & x_-
\end{array}
\right).
\end{equation}
We define:
\begin{equation}\label{E:RealNodalNF}
A(\Omega, \prec,  \tau, \sim,  \alpha, \beta, \gamma, \mathsf{wt}\bigr) = H(\Omega, \prec, \tau, \sim, \mathsf{wt}) \vee \Lambda(\overline{\Omega}, \chi, \mathsf{wt}),
\end{equation}
see diagram (\ref{E:NodalReconstruction}). It follows from Theorem \ref{T:Cartesian} that $A(\Omega, \prec,  \tau, \sim,  \alpha, \beta, \gamma, \mathsf{wt}\bigr)$ is a nodal order. 

\begin{remark} 
Let $m_\pm \in \NN$, $m = m_+ + m_-$, $F$ be a skew field and
$$
M_{m_+}(F) \times M_{m_-}(F) \xrightarrow{\jmath_{m_+, m_-}}  M_m(F), \quad (x_+, x_-) \mapsto
\left(
\begin{array}{cc}
x_+ & 0 \\
0 & x_-
\end{array}
\right).
$$
Let $g = \left(\begin{array}{cc} 0 & I_{m_-} \\ I_{m_+} & 0 \end{array}\right)$. Then the following diagram commutes:
$$
\xymatrix{
M_{m_+}(F) \times M_{m_-}(F) \ar[rr]^-{\nu} \ar[d]_-{\jmath_{m_+, m_-}} & &  M_{m_-}(F) \times M_{m_+}(F) \ar[d]^-{\jmath_{m_-, m_+}} \\
M_m(F)\ar[rr]^-{\Ad_g} &  & M_m(F)
}
$$
where $\nu(x_+, x_-) = (x_-, x_+)$. Hence, the nodal embeddings $\jmath_{m_+, m_-}$ and $\jmath_{m_-, m_+}$ are similar. This motivates the following definition.

\end{remark}

\begin{definition}\label{D:EquivNodalData} Two functions $\mathsf{wt}, \mathsf{wt}_\ast: \overline{\Omega} \lar \NN$ are called \emph{equivalent} if they agree on all elements
of $\Omega^{\mathsf{s}} \sqcup \overline{\Omega}^{\mathsf{g}}$, 
whereas $\mathsf{wt}_\ast\left|_{\overline{\Omega}^{\mathsf{d}}}\right. = 
\mathsf{wt}\left|_{\overline{\Omega}^{\mathsf{d}}}\right. \circ \vartheta$ for some involution ${\overline{\Omega}^{\mathsf{d}}} \stackrel{\vartheta}\lar {\overline{\Omega}^{\mathsf{d}}}$ such that 
$\pi_{\mathsf{d}} \vartheta = \pi_{\mathsf{d}}$.

Next, two extended tuples 
$\bigl(\Omega, \prec,  \tau, \sim,  \alpha, \beta, \gamma, \mathsf{wt}\bigr)$ and $\bigl(\Omega', \prec',  \tau', \sim',  \alpha', \beta', \gamma', \mathsf{wt}'\bigr)$
are \emph{equivalent}  if there exists a function $\overline{\Omega} \stackrel{\mathsf{wt}_\ast}\lar \NN$ equivalent to $\mathsf{wt}$, a bijection $\Omega \stackrel{\varphi}\lar \Omega'$ and a function $I_{\mathsf{co}} \stackrel{\eta}\lar \{+1, -1\}$ satisfying the following properties
\begin{enumerate}
\item[(i)] $(\varphi, \eta): \bigl(\Omega, \prec,  \tau, \sim,  \alpha, \beta, \gamma\bigr)\lar \bigl(\Omega', \prec',  \tau', \sim',  \alpha', \beta', \gamma'\bigr)$ is an equivalence of non-extended tuples in the sense of Definition \ref{D:TuplesEquivalentC}.
\item[(ii)] The diagram
$$
\xymatrix{
\overline{\Omega} \ar[rr]^-{\bar\varphi} \ar[rd]_-{\mathsf{wt}_\ast} & & \overline{\Omega}' \ar[ld]^-{\mathsf{wt}'}\\
& \NN & 
}
$$
is commutative, where $\bar\varphi$ is the bijection induced by $\varphi$. 
\end{enumerate}
\end{definition}

\begin{theorem}\label{T:Final}
The isomorphism classes of real nodal orders are in bijection with   the equivalence classes  of extended tuples $\bigl(\Omega, \prec,  \tau, \sim,  \alpha, \beta, \gamma, \mathsf{wt}\bigr)$. 
\end{theorem}

\begin{proof}
Let $A$ be a real nodal order. Then there exist a basic real nodal order $B$ and a left progenerator $P$ such that $A \cong \bigl(\End_B(P)\bigr)^\circ$. By Theorem \ref{T:NodalClassification}, there exists a tuple $\bigl(\Omega, \prec,  \tau, \sim,  \alpha, \beta, \gamma\bigr)$ such that
$B = A\bigl(\Omega, \prec,  \tau, \sim,  \alpha, \beta, \gamma\bigr)$. Let $J$ be the Jacobson radical of $B$ and set $\Lambda = B/J$. Then
$\Lambda \cong
\prod\limits_{\xi \in \overline{\Omega}} \Lambda_{\xi} = \prod\limits_{\xi \in \overline{\Omega}} \FF_{\chi(\xi)}$, where the set $\overline{\Omega}$ is given by (\ref{E:OmegaBar}) and the map $\overline{\Omega} \stackrel{\chi}\lar \Pi$ is induced by $\tau$. Moreover, there is an isomorphism of $\Lambda$-modules
$$
P/JP \cong \bigoplus\limits_{\xi \in \overline{\Omega}} \Lambda_\xi^{m_\xi}
$$
for uniquely determined multiplicities $m_\xi \in \NN$, $\xi \in \overline{\Omega}$. This defines a function
$$
\overline{\Omega} \stackrel{\mathsf{wt}}\lar \NN, \; \xi \mapsto m_\xi.
$$
For every $\xi \in \overline{\Omega}$, let $e_\xi \in B$ denote the corresponding primitive idempotent. Then
$$
P \cong \bigoplus_{\xi \in \overline{\Omega}} (B e_\xi)^{\oplus \mathsf{wt}(\xi)}.
$$
It follows that
$$
A \cong \bigl(\End_B(P)\bigr)^\circ \cong A\bigl(\Omega, \prec,  \tau, \sim,  \alpha, \beta, \gamma, \mathsf{wt}\bigr).
$$
It remains to verify that
$$
A\bigl(\Omega, \prec,  \tau, \sim,  \alpha, \beta,
\gamma, \mathsf{wt}\bigr)
\cong A\bigl(\Omega', \prec',  \tau', \sim',  \alpha', \beta',
\gamma', \mathsf{wt}'\bigr)
$$
if and only if the extended tuples $\bigl(\Omega, \prec,  \tau, \sim,  \alpha, \beta,
\gamma, \mathsf{wt}\bigr)$ and $\bigl(\Omega', \prec',  \tau', \sim',  \alpha', \beta',
\gamma', \mathsf{wt}'\bigr)$ are equivalent in the sense of Definition \ref{D:EquivNodalData}. Since
$$
A\bigl(\Omega, \prec,  \tau, \sim,  \alpha, \beta, \gamma, \mathsf{wt}\bigr) = H\bigl(\Omega, \prec, \tau, \sim, \mathsf{wt}\bigr) \vee \Lambda\bigl(\overline{\Omega}, \chi, \mathsf{wt}\bigr)
$$
and
$$
A\bigl(\Omega', \prec',  \tau', \sim',  \alpha', \beta',
\gamma', \mathsf{wt}'\bigr) = H\bigl(\Omega', \prec', \tau', \sim', \mathsf{wt}'\bigr) \vee \Lambda\bigl(\overline{\Omega}', \chi', \mathsf{wt}'\bigr),
$$
it suffices to verify that the existence of an equivalence of tuples
$$
(\varphi, \eta): \bigl(\Omega, \prec,  \tau, \sim,  \alpha, \beta,
\gamma, \mathsf{wt}\bigr) \lar \bigl(\Omega', \prec',  \tau', \sim',  \alpha', \beta',
\gamma', \mathsf{wt}'\bigr)
$$
is equivalent to the existence of an $\RR$-algebra isomorphism
$$
H\bigl(\Omega, \prec, \tau, \sim, \mathsf{wt}\bigr) \lar H\bigl(\Omega', \prec', \tau', \sim', \mathsf{wt}'\bigr)
$$
satisfying the conditions of Corollary \ref{C:Key}. As in the proof of Theorem \ref{T:NodalClassification}, this follows from Corollary \ref{C:SummaryAutomorphisms}.
\end{proof}

\end{document}